\documentclass[12pt]{amsart}

\usepackage{graphicx}

\newcommand{\tMpt}{\widetilde{\partial M}_{\tau}}

\newcommand{\I}{{\mathbf I}}

\newcommand{\szego}{Szeg\"o }

\newcommand{\kahler}{K\"ahler }

\newcommand{\N}{{\mathbb N}}
\newcommand{\R}{{\mathbb R}}
\newcommand{\C}{{\mathbb C}}

\newcommand{\Z}{{\mathbb Z}}

\newcommand{\dbar}{\bar\partial}
\newcommand{\ddbar}{\partial\dbar}

\newcommand{\E}{{\mathbf E}}

\renewcommand{\phi}{\varphi}

\newcommand{\fcal}{\mathcal{F}}
\newcommand{\gcal}{\mathcal{G}}
\newcommand{\hcal}{\mathcal{H}}

\newcommand{\jcal}{\mathcal{J}}

\newcommand{\lcal}{\mathcal{L}}

\newcommand{\ncal}{\mathcal{N}}
\newcommand{\ocal}{\mathcal{O}}
\newcommand{\pcal}{\mathcal{P}}

\newcommand{\scal}{\mathcal{S}}

\newcommand{\half}{{\frac{1}{2}}}

\newcommand{\diag}{{\operatorname{diag}}}

\renewcommand{\phi}{\varphi}
\newtheorem{maintheo}{{\sc Theorem}}
\newtheorem{mainprop}{{\sc Proposition}}
\newtheorem{mainlem}{{\sc Lemma}}

\newtheorem{maindefin}{{\sc Definition}}
\newtheorem{theo}{{\sc Theorem}}[section]
\newtheorem{cor}[theo]{{\sc Corollary}}

\newtheorem{defin}[theo]{{\sc Definition}}
\newtheorem{rem}[theo]{{\sc Remark}}

\newtheorem{lem}[theo]{{\sc Lemma}}

\newtheorem{prop}[theo]{{\sc Proposition}}

\title[Ergodicity and intersections of nodal sets and geodesics   ]{Ergodicity and intersections of nodal sets and geodesics on real analytic surfaces}

\author{Steve Zelditch }
\address{Department of Mathematics, Northwestern  University, Evanston, IL 60208, USA}
\email{zelditch@math.northwestern.edu}

\thanks{Research partially supported by NSF grants  \# DMS-0904252  and DMS-1206527.}

\begin{document}

\maketitle

\begin{abstract}   We consider the the intersections of the complex nodal set $\ncal_{\lambda_{j}}^{\C}$  of the analytic continuation of
an eigenfunction of $\Delta$ on a real analytic surface $(M^2, g)$ with the complexification of a geodesic $\gamma$.
We prove that if the geodesic flow is ergodic and if $\gamma$ is periodic and satisfies a generic asymmetry condition, then the 
   intersection points  $\ncal_{\lambda_{j}}^{\C} \cap \gamma_{x, \xi}^{\C}$   condense along the real  geodesic and become uniformly
distributed with respect to its  arc-length.  We prove an analogous result for non-periodic geodesics except that 
the `origin' $\gamma_{x, \xi}(0)$ is allowed to move with $\lambda_j$.

\end{abstract}

This article is concerned with the `complex geometry' of nodal
  sets of Laplace eigenfunctions on real analytic
Riemannian surfaces   $(M^2, g)$   with ergodic geodesic flow.  All but one of the methods 
 are valid in all dimensions, so until it is  necessary to specialize
to surfaces  we consider Riemannian
manifolds $(M^m, g)$ of general dimension $m$.  Let
$\{\phi_{j}\}$ be an orthonormal basis of eigenfunctions of the
Laplacian $\Delta$ of $(M, g)$,
$$\Delta \phi_{_j} = \lambda_j^2 \phi_{_j}, \;\;\; \langle \phi_j, \phi_k \rangle = \delta_{jk}, $$
where $\lambda_0 = 0 < \lambda_1 \leq \lambda_2 \leq \cdots$ and
where  $\langle u, v \rangle = \int_M u v dV_g$ ($dV_g$ being the
volume form). When the geodesic flow $G^t: S^*_g M \to S^*_g M$ is
ergodic on the unit (co-)tangent bundle, the `random
wave model' for eigenfunctions predicts that the nodal sets
$$\ncal_{\phi_{j}} = \{x\in M:
\phi_{j}(x) = 0\}$$ become equidistributed with respect to the
volume form in the sense that
\begin{equation} \label{PHYS} \frac{1}{\hcal^{m-1}(\ncal_{\phi_j})} \int_{\ncal_{\phi_j} }
f d \hcal^{m-1} \to \frac{1}{Vol(M, g)} \int_M f dV_g, \;\;\;
(\forall f \in C(M)). \end{equation} Here, 
$\hcal^{m-1}(\ncal_{\phi_j})$ denotes the hypersurface volume.  This prediction appears to lie far beyond the scope of
current mathematical techniques. But we aim to show that something
quite close to \eqref{PHYS} can be proved for intersections
of   complex nodal lines and complexified geodesics on real
analytic surfaces with ergodic geodesic flow. Roughly speaking, we show that in  the complex domain,  as 
$\lambda \to \infty$, the intersections of   nodal sets with generic periodic and non-periodic geodesics 
condense along the underlying real geodesic and become uniformly distributed relative to
its arclength.   Much of the proof  generalizes  to any real analytic curve on a surface or to
a real analytic hypersurface in higher dimensions, but the case of geodesics  seems
to us special and interesting enough to deserve a separate treatment; the potential generalizations
are discussed at the end of the introduction.

To state our results, we  introduce some notation.  We   recall that any real analytic
manifold $M$ admits a Bruhat-Whitney complexification $M_{\C}$,
and that for any real analytic metric $g$ all of the
eigenfunctions $\phi_j$  extend holomorphically to  a fixed open
open neighborhood $M_{\epsilon}$  of $M$ in $M_{\C}$ called a
Grauert tube. The `radius' of an open neighborhood is  measured
by the Grauert tube function $\sqrt{\rho}$ and we denote the level set $\sqrt{\rho} = \tau$ by
$\partial M_{\tau}$ for $\tau \leq \epsilon$ (see \S
\ref{BACKGROUND} for background). The  complex nodal sets are
defined by
$$\ncal_{\phi_{_j}^{\C}} = \{\zeta \in M_{\epsilon}:
\phi_{j}^{\C}(\zeta) = 0\}. $$ We plan to intersect these nodal
sets with the (image of the) complexification of an arc-length parameterized  geodesic
\begin{equation} \label{GAMMAX} \gamma_{x, \xi}: \R \to M, \;\;\;\gamma_{x, \xi}(0) = x, \;\;
\gamma_{x, \xi}'(0) = \xi \in T_x M.  \end{equation} 
 If  \begin{equation} \label{SEP} S_{\epsilon} = \{(t
+ i \tau \in \C: |\tau| \leq \epsilon\} \end{equation} then (as recalled in \S
\ref{BACKGROUND}) $\gamma_{x, \xi}$ admits an  analytic
continuation
\begin{equation} \label{gammaXCX} \gamma_{x, \xi}^{\C}: S_{\epsilon} \to M_{\epsilon}.  \end{equation}
 When we freeze $\tau$ we put
\begin{equation} \label{gammatau} \gamma_{x, \xi}^{\tau} (t) = \gamma^{\C}_{x, \xi}(t + i \tau). \end{equation}

The intersection points of $\gamma_{x, \xi}^{\C}$ and
$\ncal_{\phi_j}^{\C}$ correspond to the zeros of the pullback
$(\gamma_{x, \xi}^{\C})^* \phi_j^{\C}$. We encode this discrete
set by the measure
\begin{equation} \label{NCALCURRENT} [\ncal^{\gamma_{x, \xi}^{\C}}_{\lambda_j}] = \sum_{(t + i \tau):\; \phi_j^{\C}(\gamma_{x, \xi}^{\C}(t + i \tau)) = 0} \delta_{t + i \tau}.
\end{equation}

\begin{maindefin} \label{EQUIDEF}  Let $\scal = \{j_k\} \subset \N $ be a subsequence of the positive integers. We say that the intersection points of the complex nodal sets $\ncal_{\lambda_{j_k}}^{\C}$ 
and the complexified geodesic $\gamma_{x, \xi}^{\C}$  for the subsequence $\scal$ condense on the real geodesic and become uniformly
distributed with respect to arc-length if,
for any $f \in C_c(S_{\epsilon})$,
$$\lim_{k \to \infty}  \frac{1}{\lambda_{j_k}} \sum_{(t + i \tau):\; \phi_{j_k}^{\C}(\gamma_{x, \xi}^{\C}(t + i \tau)) = 0} f(t + i
\tau ) = \frac{1}{\pi} \int_{\R} f(t) dt.  $$
That is, $\frac{1}{\lambda_{j_k}} [\ncal^{\gamma_{x, \xi}^{\C}}_{\lambda_j}]  \to \frac{1}{\pi} \delta_0(\tau) dt  d \tau $ in the sense of
measures.
\end{maindefin}


The first result of this article  (Theorem \ref{MAINCOR}) gives
a sufficient condition on a periodic  geodesic $\gamma_{x, \xi}$ 
 for the existence of a subsequence $\scal_{x, \xi}$ of density one of the $\{\lambda_{j}\}$  for which 
$\{\ncal_{\lambda_{j_k}}\}$ satisfies the condition of Definition \ref{EQUIDEF}.
The condition is that the QER (quantum ergodic restriction) result of \cite{TZ} is valid for the geodesic. As recalled
in \S \ref{QER}, the  QER result states that restrictions of eigenfuntions $\phi_{j_k} |_{\gamma_{x, \xi}}$ to a (real) geodesic
are quantum ergodic along $\gamma_{x, \xi}$ as long as
 $\gamma_{x, \xi}$ is {\it asymmetric} (as a hypersurface of $M$) with respect to the
geodesic flow. The asymmetry  condition (Definition \ref{ANC})  means that geodesics $\alpha(t): \R \to M$ with $\alpha(0) \in \gamma_{x, \xi}$ 
(i.e. the trace of $\gamma_{x, \xi}$) and the geodesic $\alpha^*(t)$ such that $\alpha(0) = \alpha^*(0)$ and with
$\alpha^{* '}(0) $ equal to the reflection through $T_{\alpha(0)} \gamma_{x, \xi}$ of $\alpha'(0)$ almost never return
to $\gamma_{x, \xi}$ at the same time and the same point. Since geodesics are  hypersurfaces only
when $\dim M = 2$, the result on intersections of periodic geodesics and nodal sets  is only proved in that dimension.
Results independent of QER hold in all dimensions.

The second result ( Theorem \ref{MAINCORnonper}) is an analogous result for  non-periodic geodesics, such  as Birkhoff regular ones. The result is somewhat weaker due  to the non-compactness of non-periodic geodesics and the resulting
problems with  escape of mass at parameter time infinity.


\subsection{A key Lemma} Before stating the Theorems precisely, we state a key Lemma which is valid
in all dimensions and  which reduces the equidistribution of zeros in the ergodic case to growth estimates. 
 It is ultimately based on the key Proposition \ref{LL} which we state
after some further preliminaries. 

The starting point is 
 the Poincar\'e-Lelong formula, according to which we may express the
current of summation over the intersection points in \eqref{NCALCURRENT} in the form,
\begin{equation}\label{PLL}  [\ncal^{x, \xi}_{\lambda_j}] = \frac{ i}{\pi} \ddbar_{t + i
\tau} \log \left| \gamma_{x, \xi}^* \phi_{\lambda_j}^{\C} (t + i \tau)
\right|^2. \end{equation}   This formula holds for the complexification of any real analytic curve. It follows from
\eqref{PLL} that the main step in the proofs of all the theorems is to obtain the asymptotics of  the sequence
\begin{equation} \label{vxxi} v_j^{x, \xi} : = \frac{1}{\lambda_{j} } \log \left| \gamma_{x, \xi}^* \phi_{\lambda_j}^{\C} (t + i \tau)
\right|^2, \;\;  \; ((x, \xi)  \in S^*M) \end{equation}
of subharmonic functions on a strip $S_{\epsilon} \subset \C$.

A key Lemma is the following compactness result, which 
combines   a standard compactness Lemma due to Hartogs, H. Cartan
and L. H\"ormander with a stronger conclusion that is ultimately based on Proposition \ref{LL} below. We use
the notation $v^*$ for the USC (upper semi-continuous) regularization of $v$. For background
we refer to \cite{Ho2} (Theorems 3.2.12-3.2.13).

\begin{mainlem} \label{HARTOGSIntro} For any compact analytic Riemannian manifold $(M, g)$,  and any $(x, \xi) \in S^* M$,
the family of subharmonic functions 
$$\fcal^{x, \xi}: = \{ v_j^{x, \xi} (t + i \tau), \; j = 1,  2, \dots \} $$ on the strip $S_{\epsilon}$   is precompact in 
$L^1_{loc}(S_{\epsilon})$ as long as 
it does not converge  uniformly to $-
\infty$ on all compact subsets of $S_{\epsilon}$. Moreover:

\begin{itemize}

\item For all $(x, \xi)$,
$\limsup_{k \to \infty} v_{k}^{x, \xi}(t + i \tau)
 \leq 2 |\tau| $.

\item Let  $\{v_{j_k}^{x, \xi}\}$ be any subsequence of  $\{v_j^{x, \xi}\}$ with a unique $L^1_{loc}$ limit
$v$ on $S_{\epsilon}$ and let $v^*$ be its USC regularization.
 Then if  $v^* <  2 |\tau| - \epsilon$ on an open set $U \subset S_{\epsilon}$  then  $v^*  \leq 2 |\tau| - \epsilon$
for $\tilde{U} = \bigcup_{t \in \R} (U + t)$  and 
\begin{equation} \label{YBAD} \limsup_{k \to \infty} v_{j_k} \leq |\tau| - \epsilon \;\;\; \mbox{on}\;\; \tilde{U}. \end{equation}

\end{itemize}

\end{mainlem}

The upper bounds follow from the global upper bound
\begin{equation} \label{UB} \limsup_{k \to \infty} \frac{1}{\lambda_j} \log |\phi_{j_k}(\zeta)|^2 \leq  2 \sqrt{\rho}(\zeta)\end{equation}
everywhere on $\partial M_{\tau}$ proved in \cite{Z};   we review it in \S \ref{LWLSECT}. However, it is not generally true that if a limit $v^*$
is $<  2|\tau| - \epsilon$ on some open set then it is globally $< 2 |\tau| - \epsilon$ on $\R$. This is where
Proposition \ref{LL} is used.

Since there is no unique choice of origin along (the trace of) the parameterized  geodesic $\gamma_{x, \xi}$
it is natural to  consider the enlarged family
 \begin{equation} \label{FAMILYs} \fcal^{x, \xi}_{\R} := \bigcup_{s \in \R}
 \fcal^{G^s(x, \xi)}\ \end{equation}  of translates of  $\gamma_{x, \xi}^* \phi_j$ for $j = 1, 2, \dots$. The 
compactness result of Lemma \ref{HARTOGSIntro} generalizes to this family; we refer to Lemma 
\ref{HARTOGSM} for the statement we need.

\subsection{Statement of results for asymmetric periodic geodesics on surfaces}

The first result pertains to  periodic  geodesics on surfaces which satisfy the  asymmetry condition 
(Definition \ref{ANC}). 
 The asymmetry condition is needed
to rule out obvious counter-examples such as when 
$\gamma_{x, \xi}$ is the fixed point set of an isometric involution; then ``odd" eigenfunctions under the 
involution will vanish everywhere  on the geodesic. The  asymmetry condition  originated in    \cite{TZ}  as the condition
that $\gamma_{x, \xi}$ have the QER (quantum ergodic restriction) property, i.e  that
there exists a full density set of eigenfunctions $\{\phi_{j_k}\}$
which are quantum ergodic when restricted
to $\gamma_{x, \xi}$  \cite{TZ}; see also \cite{DZ}.  The asymmetry  of periodic geodesics on hyperbolic 
surfaces is studied in \cite{TZ} and the discussion is almost the same for any surface of negative curvature. 
Hence we do not discuss existence of asymmetric periodic geodesics in this article.

\begin{maintheo}\label{MAINCOR} Let $(M^2, g)$ be a real analytic Riemannian surface  with ergodic
geodesic flow. Let $\gamma_{x, \xi}$ be a periodic geodesic satisfying the assymetry QER hypothesis of Definition \ref{ANC}.   Then there
exists a subsequence of eigenvalues $\lambda_{j_k}$ of density one
such that the equi-distribution result in Definition \ref{EQUIDEF} holds. 
\end{maintheo}

The main Proposition is:

\begin{mainprop} \label{MAINPROP} (Growth saturation) If  $\gamma_{x, \xi}$ is a periodic geodesic which satisfies the  QER asymmetry condition (Definition \ref{ANC})  along compact arcs,  then there exists a subsequence $\scal_{x, \xi}$  of density
one so that, for all $\tau < \epsilon$,   $$\lim_{k \to \infty} \frac{1}{\lambda_{j_k}} \log \left| \gamma_{x, \xi}^{\tau *} \phi_{\lambda_{j_k}}^{\C} (t + i \tau)
\right|^2 = 2 |\tau|\;\;\; \mbox{in}\;\;L^1_{loc} (S_{\tau}).  $$
The subsequence $\scal_{x, \xi}$ is the ergodic sequence along $\gamma_{x, \xi}$ given by Theorem \ref{QER}.

\end{mainprop}

Proposition \ref{MAINPROP} immediately implies Theorem \ref{MAINCOR} since we can apply $\ddbar$ to the
$L^1$ convergent sequence $\frac{1}{\lambda_{j_k}} \log \left| \gamma_{x, \xi}^* \phi_{\lambda_{j_k}}^{\C} (t + i \tau)
\right|^2 $ to obtain a weakly convergence sequence of measures tending to $\ddbar |\tau|$. This Proposition has an analogue for any real analytic curve but the exact formula
is special to geodesics and arises because complex geodesics are isometric embeddings to Grauert tubes
(see \S \ref{KISOMEM}).  In general, the growth rates of restrictions depend on the curve.

\subsection{Statement of results for non-periodic geodesics}

We next  consider non-periodic geodesics. 
The non-compactness of $\R$ may allow the  $L^2$ mass of the restricted
eigenfunctions (in the real or complex domain)  to escape `to infinity' along the parameter interval $\R$ of the geodesic as $\lambda_j \to \infty$. That is, 
  $|\phi_j^{\C} |_{\gamma^{\tau}_{x, \xi}}$  might achieve growth saturation only  on  intervals $I_j \subset \R$ which 
get translated  to infinity
as $j \to \infty$. Viewed on the compact manifold $M$, the intervals have limit sets but  they might consist of  arcs along different geodesics, i.e. the saturating mass might jump in the limit to another geodesic.  

To gain some partial compactness, we consider the two-parameter family
\eqref{FAMILYs}  of
restrictions $\gamma_{G^s(x, \xi)}^* \phi_j$ as $s, \lambda_j$ vary.  For fixed $\lambda_j$ this is the family 
of translates of $\gamma_{x, \xi}^* \phi_j$. Of course, this  family is non-compact in  $C_b(\R)$ since  $\gamma_{x, \xi}^* \phi_j$ is
not an almost-periodic function.

\begin{maintheo}\label{MAINCORnonper} Let $(M^2, g)$ be a real analytic Riemannian surface  with ergodic
geodesic flow. Let $\gamma_{x, \xi}$ be a non-periodic geodesic satisfying the assymetry QER hypothesis of Definition \ref{ANC}.   Then there
exists a subsequence of eigenvalues $\lambda_{j_k}$ of density one and a sequence $\{N_k\} \subset \R$
such that for any $f \in C_c(S_{\epsilon})$,
$$\lim_{k \to \infty} \sum_{(t + i \tau):\; (\gamma_{(x, \xi)}^{*}\phi_{j_k}^{\C}(t + N_k + i \tau)) = 0} f(t + i
\tau ) =  \frac{1}{\pi}  \int_{\R} f(t) dt.  $$
\end{maintheo}
Thus, we obtain a result parallel to that of Theorem \ref{MAINCOR} except that
 we may have to translate the  origin  $x$   $\gamma_{x, \xi}$  unbounded parameter distances
along the geodesic. 

These  concentration- equidistribution results are   `restricted'
versions of the result of \cite{Z}, which states that for real
analytic $(M, g)$ with ergodic geodesic flow,
\begin{equation} \label{Z} \langle \psi, [\frac{1}{\lambda_j}
\ncal_{\phi_{\lambda_j}^{\C}}]\rangle \to \frac{1}{\pi} \int_{M_{\tau}} \psi
\omega^{m-1} dd^c \sqrt{\rho} \end{equation}  for a density one
subsequence of ergodic eigenfunctions.  Here, $\omega = dd^c \rho$
is the \kahler metric on the Grauert tube induced by $g$ (see
\cite{GS1,LS} and \S \ref{BACKGROUND}). An important point to
observe is that $dd^c \sqrt{\rho}$ is singular along the real
domain, indicating that complex zeros concentrate along the
totally real submanifold $M$. Our results show   that the
singularity is magnified under restriction  to asymmetric geodesics, indeed
it becomes a delta function along the real geodesic. We also refer to
\cite{NV,SZ} for earlier papers studying distribution of complex zeros of
ergodic eigensections. 

It follows that there exist many ``nearly real'' intersections of a complex geodesic 
with the complex nodal line when the geodesic flow is ergodic (i.e. zeros of $\gamma_{x, \xi}^*\phi_j^{\C} $
whose imaginary parts tend to zero with $\lambda_j$). It would be very  interesting  to know  the proportion of ``truly real"
intersection points among these nearly real ones.  There are very few lower bounds known at present on
the number of  real intersection points, except in special cases such as separation
of variables eigenfunctions or eigenfunctions on  flat tori \cite{BR} or  special eigenfunctions and the 
geodesics on the modular
hyperbolic surface \cite{GRS}.    

\subsection{Discussion  of the proofs}

The proofs  involve several principles which played no role in the global
result \eqref{Z}. Some hold in much greater generality and some are specific to geodesics. At the end of the introduction
we discuss the potential generalizations.

One of the main ingredients in the proof is an invariance principle for restrictions
to geodesics in the complex domain that is a simple kind of QUER (quantum uniquely
ergodic restriction) principle.  The main statement (Lemma \ref{LL})  proves the translation invariance of the limit measures of 
$L^2$ normalizations of $|\gamma_{x, \xi}^{*} \phi_j^{\C}(t + i \tau)|^2$ 
along intervals in each  horizontal line of $S_{\tau}$.  Intuitively, it is the restricted version of the standard fact that Wigner measures of eigenfunctions
are almost invariant under the geodesic flow. Since we are restricting to a single geodesic, the result should be translation invariance
of the limit measures. But we obviously need to normalize  $\gamma_{x, \xi}^{*} \phi_j^{\C}(t + i \tau)$ along horizontal lines to
obtain a bounded family of measures and its limit measures.  The non-compactness of $\R$ in the case of non-periodic geodesics forces us
to work on compact sub-intervals.

In the case of periodic geodesics, we can normalize $ \gamma_{x, \xi}^{*} \phi_j^{\C}(t + i \tau)$ for each $\tau$ by dividing by its $L^2$
norm for $t \in [0, L]$ where $L $ is the period of $\gamma_{x, \xi}$ in the real domain (hence also in the complex domain).  When
 $\gamma_{x, \xi}$ is a non-periodic geodesic, there is no canonical choice of normalization and therefore we consider
all possibly choices.
When we pull back under $\gamma_{x, \xi}^{\tau}$, we consider all possible  renormalizations along intervals $I$ as follows:

\begin{maindefin}  \label{Ujdef}  Let  $I_{\tau} \subset \partial S_{\tau} = \{ t + i \tau:  t \in I\}$ be the indicated segment
of $\partial S_{\tau}$. Then define
\begin{equation}  U_j^{I_{\tau}, x, \xi}  := 
 \frac{\gamma_{x, \xi}^{\tau *} \phi_j^{\C}|_{\partial S_{\tau}}}{||\gamma_{x, \xi}^{\tau *} \phi_j^{\C}|_{\partial S_{\tau}}||_{L^2(I_{\tau}, dt)}}, \;\; (\mbox{i.e.}\;\; \int_{I_{\tau}} | U_j^{I_{\tau}, x, \xi} |^2 dt = 1), \end{equation}
where 
$$||f||_{L^2(I_{\tau})} =  \int_{t \in I} |f(G^{t + i \tau}(x, \xi)) |^2 dt. $$


\end{maindefin}

We then associate   Wigner measures to normalized complexified eigenfunctions. As will be explained below,
in the complex domain the relevant theory of pseudo-differential operators is the Toeplitz calculus of \cite{BoGu}. This reflects
the fact that restricted  complexified eigenfuntions concentrate microlocally on the tangent directions to $\gamma_{x, \xi}$. Hence Wigner measures
are defined simply by multiplication operators along $\partial S_{\tau}$: 

\begin{maindefin}  \label{WIGNER} Let $a \in C_c^{\infty} (\partial S_{\tau})$ and set

\begin{equation}  \int_{S_{\tau}} a(t) d W_j^{I_{\tau}, x, \xi}(t) 
: = \langle  a U_j^{I_{\tau}, x, \xi}, U_j^{I_{\tau},  x, \xi} \rangle  = \int_{\partial S_{\tau}} a |U_j^{I_{\tau}, x, \xi} |^2 dt \end{equation}
\end{maindefin}

Our aim is then to  determine the weak* limits of 
$$ \left|  U_j^{I_{\tau}, x, \xi}
\right|^2\;\; \mbox{on}\;\; I_{\tau}\;\; 
\mbox{and on general compact line segments of}  \;\;\; \partial S_{\tau, T}. $$
Since $ \left|  U_j^{I_{\tau}, x, \xi}
\right|^2$ is normalized to have mass one on $I_{\tau}$ it forms a pre-compact family of probability measures.  The following Proposition asserts that 
the Wigner distributions
are asymptoticaly  invariant under translation.

\begin{mainprop} \label{LL} (Lebesgue limits)  Let $(M, g)$ be a real analytic Riemannian manifold
of any dimension $m$, Let $(x, \xi) \in S^* M$ be any point. Then  as long 
as $\gamma_{x, \xi}^* \phi_j \not=  0$ (identically),  the 
 sequence $\{| U_j^{I_{\tau}, x, \xi}|^2\}$  is QUE  on $\I_{\tau}$ with limit measure given by normalized  Lebesgue measure on
$\I_{\tau}$. 
That is, for any  $a \in C_c^{\infty}(I_{\tau})$,   we have
$$\lim_{j \to \infty} \int_{\I_{\tau}}  a (t ) d W_{j}^{I_{\tau}, x, \xi} = 
\frac{1}{|\I_{\tau}|} \int_{ \I_{\tau}}  a ( s)  ds. $$

\end{mainprop}

The proof of Proposition \ref{LL} uses the Toeplitz Fourier integral operator calculus of Boutet de Monvel-Guillemin \cite{BoGu}. Toeplitz
operators arise  in the complex domain because the restriction $\gamma_{x, \xi}^{\tau *} V_{\tau}^t$ of the wave
group $V^t_{\tau}$ on $\partial M_{\tau}$  is a Toeplitz Fourier integral operator (we refer to \S \ref{WG} for the
notation). This is the analogue in the complex domain of the operator $W = \gamma_{x, \xi}^* U(t)$  studied in \cite{TZ}.
In the real domain this operator is a Fourier integral operator with a one-sided fold singularity; in the complex domain
the analogous operator is of a very different type: it is  a Toeplitz Fourier integral which microlocally lives on the tangent
space to the geodesic.
The main point of the proof is to show that the  Wigner distributions $U^{I, x, \xi}_j$ are almost invariant under time translation. 
 But the only translation  invariant measures on $\R$  are constant multiples of Lebesgue measure.
Since  we  normalized the Wigner distributions
to have integral $1$ over $I_{\tau}$, the constant must be one on that interval.

The behavior of the local  mass on general intervals is not clear apriori when $\gamma_{x,\xi}$ is non-periodic, especially
at parameter distances $t$ exceeding the `Eherenfest time'  $\log \lambda_j$, where the remainders in Egorov type
theorems break down. 
The weak * limits cannot be deduced
from those on  $\partial M_{\epsilon}$ (which were studied in \cite{Z}) since weak* convergence  is not preserved by restriction to sets of
measure zero.

Proposition  \ref{LL} combines with Lemma \ref{HARTOGSIntro} as follows: 
  If any of  limit is $< \tau - \epsilon$ on 
an open set, then by Proposition  \ref{LL} it has to be $< \tau - \epsilon$ on all of $\partial S_{\tau}$. Otherwise,
the normalizations of $U_j^{I, x, \xi}$ would have different exponential orders and Proposition \ref{LL}
could not be true for every interval $I$. We rely on the fact that Proposition \ref{LL} holds
simultaneously for all of the  normalized pullbacks
in Definition \ref{Ujdef}.

We now sketch the proof of Theorem \ref{MAINCOR} on periodic geodesics and of  Proposition \ref{MAINPROP} to highlight the differences between restriction to 
geodesics in the real and complex domains and to indicate the kinds of new phenomena that are needed in the proof.
To prove Proposition \ref{MAINPROP}  in the case of periodic $\gamma_{x, \xi}$, we first prove an integrated version for $L^2$ norms.

\begin{mainlem} \label{L2NORMintro} Let  $\gamma_{x, \xi}$ be a periodic geodesic of period $L.$   Assume that $\{\phi_j\}$ satsifies QER along the
periodic geodesic $\gamma_{x, \xi}$. Let $||\gamma_{x, \xi}^{\tau*} \phi_j^{\C}||^2_{L^2(\partial S_{\tau})}$ be the $L^2$-norm
of the complexified restriction of $\phi_j$ to a period cell  $\partial S^L_{\tau}$ of $\partial S_{\tau}$. Then,
$$\lim_{\lambda_j \to \infty} \frac{1}{\lambda_j} \log ||\gamma_{x, \xi}^{\tau*} \phi_j^{\C}||^2_{L^2(\partial S^L_{\tau})}
= 2 |\tau| .$$
\end{mainlem}

To prove Lemma \ref{L2NORMintro}, we study the      orbital Fourier series of $\gamma_{x, \xi}^{\tau*} \phi_j$
and of its complexification. The orbital Fourier coefficients are 
$$\nu_{\lambda_j}^{x, \xi}(n) = \frac{1}{L} \int_0^{L} \phi_{\lambda_j}(\gamma_{x, \xi}(t)) e^{- \frac{2 \pi i n t}{L}} dt, $$
and the orbital Fourier series is 
\begin{equation} \label{PER} \phi_{\lambda_j}(\gamma_{x, \xi}(t) )= \sum_{n \in \Z}  \nu_{\lambda_j}^{x, \xi}(n)  e^{\frac{2 \pi i n t}{L}}. 
\end{equation}
Hence the analytic continuation of $\gamma_{x, \xi}^{*} \phi_j$  is given by 
\begin{equation} \label{ACPER} \phi^{\C}_{\lambda_j}(\gamma_{x, \xi}(t + i \tau) )= \sum_{n \in \Z}  \nu_{\lambda_j}^{x, \xi}(n)  e^{\frac{2 \pi i n (t + i \tau)}{L}}. \end{equation}
By the Paley-Wiener theorem for Fourier series, the  series converges absolutely and uniformly for $|\tau| \leq \epsilon_0$. 
The growth rate of $\phi^{\C}_{\lambda_j}(\gamma_{x, \xi}(t + i \tau) )$ is thus intimately related to 
the joint asymptotics  of the Fourier coefficients $\nu_{\lambda_j}^{x, \xi}(n)$ in $(\lambda_j, n)$.  We use the
QER hypothesis in the following way:

\begin{mainlem} \label{FCSAT}  Suppose that $\{\phi_{\lambda_j}\}$ is QER along the periodic geodesic $\gamma_{x, \xi}$.
Then for all $\epsilon > 0$, there exists $C_{\epsilon} > 0$ so that
$$\sum_{n: |n| \geq (1 - \epsilon) \lambda_j}  |\nu_{\lambda_j}^{x, \xi}(n)|^2 \geq  C_{\epsilon}. $$

\end{mainlem}

 Lemma \ref{FCSAT}  implies Lemma \ref{L2NORMintro} since it implies that for any $\epsilon > 0$,
$$\sum_{n: |n| \geq (1 - \epsilon) \lambda_j}  |\nu_{\lambda_j}^{x, \xi}(n)|^2 e^{-2 n \tau}  \geq  C_{\epsilon} e^{2\tau(1 - 
\epsilon) \lambda_j}. $$

In essence, we prove  the lower bound in Proposition \ref{MAINPROP}  in the ergodic case by showing that all of
the Fourier coefficients in the allowed energy region $|n| \leq \lambda_j$ are of uniformly large size. Since the top
frequency term dominates and its Fourier coefficient is large, $\gamma_{x, \xi}^*\phi_j^{\C}$ must have maximal
growth. 

The argument sketched above only proves the desired logarithmic growth law of Proposition \ref{MAINPROP}
for $L^2$-norms.   Proposition \ref{LL}  improves it
 to give the local $L^1$-convergence statement of   Proposition \ref{MAINPROP}.

The proof of Theorem \ref{MAINCORnonper} follows the same general outline but is more complicated
for two reasons: first, the Fourier transform of $\gamma_{x, \xi}^* \phi_j$ is the Fourier transform of an
$L^{\infty}$ function and not an $L^2$ function. It can be shown that it is not even a measure and so we 
cannot speak of the `size' of the Fourier coefficients. Hence it
has to be multiplied by a decaying analytic function before the Fourier transform is taken. And as mentioned above,
the Fourier coefficients may saturate the growth bounds  somewhere on $\R$ for each $j$ but the location of the
saturating mass may escape
to infinity along the parameter interval. This explains why we may need to introduce translations $N_j$ into the times.


\subsection{Generalizations} 

There are several natural generalizations of   intersection problems for geodesics
and nodal hypersurfaces to consider: (i) to general real analytic curves $C \subset M^2$
of a real analytic surface; (ii) to general real analytic hypersurfaces $H \subset M^m$ in any dimension;
(iii) to generic or random geodesics in all dimensions. We plan to investigate the generalizations in a subsequent article.

The generalizaton (i) is simplest, since  most of the techniques and results of this article apply to any real analytic curve $C \subset M^2$. 
The main one which does not is Proposition \ref{LL}, in the cases  when the curve is not a geodesic. Moreover,
the unit speed parametrization of $C$ no longer complexifies to an isometric embedding of Grauert tubes,
and it is not as simple to find the growth rate of $\frac{1}{\lambda_j} \log |\phi_j^{\C}(\gamma_C^{\C}(t + i \tau)|$;  it does
not equal   $|\tau|$ when $C$ is not a geodesic, and  one does not get the same concentration of complex zeros along the real points, i.e. there exists
an additional smooth component to the limit distribution of complex zeros. 

The  additional complication in (ii) is that nodal hypersurfaces intersect other hypersurfaces $H$  in
codimension 2 submanifolds rather than discrete points. Hence the limit measure
will be  a locally $L^1$ measure on an a complex $(n-1)$-dimensional  parameter domain.  Instead
of expanding the restriction as a Fourier series or integral we would need to use  eigenfunctions of the hypersurface
Laplacian. 

The generalization (iii) is the most difficult,  and it is not clear at the moment whether a generalization
of Theorem \ref{MAINCOR} to periodic geodesics in higher dimensions exists. We cannot appeal to the QER result of \cite{TZ} in this case. 
The QER result is a quantum analogue of the fact that  unit (co-)vectors in $S^*_H M$, i.e. with footpoint on $H$, form a cross
section to the geodesic flow when $H$ is a hypersurface and thus the first return map is ergodic.  When 
$\dim H < \dim M -1$, $S^*_H M$ is no longer a cross section and the proof in \cite{TZ} does not apply.  It is not clear
that the Fourier coefficients always  have uniformly the same size for $|n| \leq \lambda.$ It is tempting to relate the $L^2$ mass of restricted eigenfunctions to their global mass when the geodesic is
Birkhoff regular (i.e. uniform) but the results so far do not improve on Theorem \ref{MAINCORnonper}.
   However it is possible
to prove a somewhat weaker version of Theorem \ref{MAINCOR} for random geodesics; we postpone the proof
to a subsequent article. 

\subsection{Acknowledgements} Thanks to  Simon Marshall for many helpful conversations on earlier drafts of this article,
and to John Toth for collaboration on related problems in \cite{TZ} and elsewhere.

\section{\label{BACKGROUND}  Geometry of geodesics and  Grauert tubes}

In this section, we discuss the geometry of geodesics and their complexifications in Grauert tubes.  
We need to relate the holomorphic extension of $\gamma_{x, \xi}$  to the Hamilton flow of the Grauert 
tube function.  The relations are valid in all dimensions and so we consider a general real analytic
Riemannian manifold $M^m$ of dimension $m$.

\subsection{Geodesic flow in the real domain}

A potentially confusing point is  that there are two geodesic flows in the real domain, and both are
relevant to the microlocal analysis of wave groups.  Below we denote by $\pi: T^*M \to M$
the standard projection. 

\begin{itemize}

\item The geometer's geodesic flow: the Hamilton flow of the Hamiltonian $H(x,\xi) = |\xi|_g^2 
= \sum_{i, j = 1}^d g^{ij}(x)  \xi_i \xi_j$ on $T^*M $. We denote its flow by $g^t: T^*M \to T^* M$. 
The usual exponential map is defined by
$\exp_x \xi = \pi g^1(x, \xi)$.  On the zero section $0_M \subset T^* M$ it is the identity map.

\item The homogeneous geodesic flow $G^t: T^*M - 0 \to T^* M - 0$  (the bicharacteristic flow of the wave operator): It is the Hamiltonian
flow generated by $H_1(x, \xi) = |\xi|_g = \sqrt{H}$.  It is homogeneous in the sense that
$G^t(x, r \xi) = r G^t(x, \xi)$. It is not defined on $0_M$.

\end{itemize}

\subsection{Grauert tubes}

For background on Grauert tube geometry we refer to
\cite{GS1,GS2,LS}.  A real analytic manifold $M$ always possesses
a complexification $M_{\C}$, i.e. a complex manifold  of which $M$
is a totally real submanifold.  A real analytic metric $g$ then
determines a canonical plurisubharmonic function $\rho_g$ on
$M_{\C}$ whose square root $\sqrt{\rho}$ is known as the Grauert
tube function.  In fact, $\rho (\zeta) = - r^2_{\C}(\zeta,
\bar{\zeta})$ where $r^2_{\C}$ is the holomorphic extension of the
distance-squared function from a neighborhood of the diagonal in
$M \times M$. The $(1,1)$- form $\ddbar \sqrt{\rho}$ has rank $m -
1$ on $M_{\epsilon} \backslash M$, and the leaves of its null foliation (the
`Monge-Amp\`ere' or Riemann foliation)  are the traces of the complexified
geodesics  $\gamma_{\C}(t + i \tau)$. 
The Grauert tubes are defined by
$$M_{\epsilon} = \{\zeta \in M_{\C}: \sqrt{\rho}(\zeta) \leq
\epsilon\}. $$ We define the  K\"ahler form $\omega = \omega_g$  of $M_{\epsilon}$  by
\begin{equation} \label{OMEGAG}  \omega = i^{-1} \ddbar \rho. \end{equation}  The unusual
sign convention (making the \kahler form a  negative rather than positive) (1,1) form) is adopted from \cite{GS1}.
  In terms of  the real operators (with $J^*$ the complex structure operator on 1-forms),
$$d^c = \frac{i}{4 \pi} (\dbar - \partial) =  - \frac{1}{4 \pi} J^*  d , \;\; \;\; dd^c = - d^c d= \frac{i}{2 \pi} \ddbar, $$
we have
\begin{equation} \label{DDC} \omega = - 2 \pi dd^c \rho.  \end{equation}

 Following \cite{GLS}, we define the imaginary time  complexified exponential map by \begin{equation}
\label{EXP} E: (x, \xi) \in B_{\epsilon}^*M \to \exp_x^{\C}
\sqrt{-1} \xi \in M_{\epsilon}. 
\end{equation} 
The following Lemma records the way that $E$ transfers objects between $M_{\epsilon}$ and the co-ball bundle
$B^*_{\epsilon} M$ of radius $\epsilon$ in $T^* M$.

\begin{lem} \label{EPSI} Let $\alpha = \xi \cdot dx$ resp.  $\omega_{T^* M}$ be the canonical 1-form, resp. symplectic form, of
$T^* M$. Then  $E: ( B_{\epsilon}^*M, \omega_{T^* M}) \to (M_{\epsilon}, \omega)$  is a symplectic  diffeomorphism from the co-ball bundle of
radius $\epsilon$ to $M_{\epsilon}$. It has the properties:
\begin{itemize}

\item   $E^* |\xi|_g^2 = \rho_g$ and $E^* |\xi|_g =
\sqrt{\rho}$. 

\item  $E^* \alpha = \Im \dbar \rho = d^c \sqrt{\rho}$ and   $E^* \omega = \omega_{T^* M}.$
\end{itemize}

\end{lem}

\begin{proof}  This is a reformulation of some results of \cite{GS1, GLS, LS}.  The Theorem on p. 568 of \cite{GS1}
states that $\alpha = \Im \dbar |\xi|_g^2$ on $T^* M$ where $\dbar$ is with respect to the ``adapted complex
structure" on $B^*_{\epsilon} M$. Using a  theorem of Kostant-Sternberg, they produce  a unique diffeomorphism
$\psi: (B^*_{\epsilon} M, \omega_{T^* M} ) \to (M_{\epsilon}, \omega).$ It thus suffices to show that $E = \psi$. 
This follows from  the  uniquness argument in \cite{GLS} Proposition 1.7 and Theorem 1.8. That is, $ \psi \circ E^{-1}$
is a biholomorphic map of $M_{\epsilon}$ which equals the identity on the totally real submanifold $M \subset M_{\epsilon}$,
and therefore must be the identity map.

\end{proof}

Since $(M_{\tau}, \omega_{\rho})$ is a symplectic manifold, we can consider Hamiltonian flows of $\rho$
and $\sqrt{\rho}$.  We denote the Hamilton flow of any function $H$ with respect to $\omega = \omega_{\rho}$  by $\exp t \Xi^{\omega}_H$. When $\omega$ is understood, we 
 denote by $\Xi_{\sqrt{\rho}}$  the Hamilton vector field with respect to $\omega$.   The following Lemma asserts
that $E$ interwines (both) geodesic flows on $B_{\epsilon}^* M - 0$ with Hamilton flows on $M_{\epsilon}$.
\begin{lem} \label{HAMID}
\begin{equation}  \label{EINT} E \circ  G^t = \exp t \Xi_{\sqrt{\rho}} \circ E, \;\;\;\; E \circ g^t = \exp t \Xi_{\rho} \circ E. \end{equation}
\end{lem} 

\begin{proof}
It follows from Lemma \ref{EPSI} that  $E^* \sqrt{\rho} = |\xi|$ and that $E^* \omega = \omega_{B^* M}$. Hence
$E$ intertwines the associated Hamiltonian flows. Similarly for the Hamiltonian flow of $\rho$.

\end{proof}

We recall   Liouville measure on each sphere bundle
$S^*_{\epsilon} M$ of radius $\epsilon$ is given by $d\mu_L =
\alpha \wedge \omega_{T^* M}^{d-1}$. Under the complexified
exponential map $E$, Liouville measure pulls  to $\partial
M_{\epsilon}$ to
$$E^* d\mu_L = d^c \sqrt{\rho} \wedge \omega^{m-1}. $$

\subsection{\label{CXHAM} Complexified  Hamiltonian flow}

In addition to its \kahler form $\omega_g \in \Omega^{1,1}(M_{\C})$, the Grauert tube carries the complex
holomorphic metric
\begin{equation} \label{gC} g_{\C} = \sum_{i, j= 1}^n g^{ij}(\zeta) d \zeta_i \otimes d \zeta_j  \end{equation}
in local holomorphic coordinates on $M_{\C}$. 
The complexified geodesic flow
\begin{equation} \label{GT} g_{\C}^{t + i \tau}: T^* M_{\C} \to T^* M_{\C}\end{equation} is the partially defined 
 Hamiltonian flow of the holomorphic Hamiltonian,
\begin{equation} \label{HC} H_{\C}(\zeta, \xi) = \sum_{i, j= 1}^n g^{ij}(\zeta)  \xi_i \xi_j \end{equation}
on $T^* (M_{\C})$ with respect to its holomorphic symplectic form 
$$\sigma = \sum_j d \zeta_k \wedge d \xi_k$$
where $\xi_k$ are the coordinates of a (1,0)-form in the basis $d \zeta_k$. 
The Hamiltonian vector field is then
\begin{equation} \label{XI} \Xi_{H} = \sum_{i, j= 1}^n g^{ij}(\zeta)  \xi_i \frac{\partial}{\partial \zeta_j}
- \sum_{i, j= 1}^n \frac{\partial g^{ij}(\zeta) }{\partial \zeta_k}  \xi_i \xi_j \frac{\partial}{\partial
\xi_k}.  \end{equation}
 In these coordinates the canonical one-form on 
$T^* M_{\C}$ is given by 
\begin{equation} \alpha_{T^* M_{\C} }= \sum_j \xi_j d \zeta_j. \end{equation}
We also denote by $\pi: T^* M_{\C} \to M_{\C}$ the natural projection. 

The following Lemma  relates $\gamma_{x, \xi}^{\C}(t + i \tau)$ to the holomorphic
geodesic flow. 

\begin{lem} \label{pi} Let $(x, \xi) \in T^* M$ and let   $\gamma_{x, \xi}(t + i \tau) : \R \to M_{\tau}$  be the analytic continuation of $\gamma_{x, \xi} (t)
= \exp_x t \xi$. Then:

\begin{itemize}

\item  $\gamma_{x, \xi}(t + i \tau) = \pi g_{\C}^{t + i \tau} (x, \xi). $ 

\item  $E(x, \xi) = \pi g^1_{\C} (x, i \xi). $

\item $g_{\C}^{t + i \tau} = g_{\C}^{i \tau} g_{\C}^t. $

\item  $\gamma_{x, \xi}(t + i \tau) = 
\pi G^{t + i \tau}(x, \xi) = \pi G^{i \tau} G^t(x, \xi) = \gamma_{G^t(x, \xi)} (i \tau). $

\end{itemize} 

\end{lem}

Indeed, these identities hold for real $t$ and therefore analytically continue to complex $t$.

\subsection{\label{HAMFLOWREL} Hamiltonian flow of $\sqrt{\rho}$ and complexified geodesics}

In this section we relate $\gamma_{x, \xi}^{\C}(t + i \tau)$ to the Hamilton orbits of $\Xi_{\rho}$. Although
both arise from Hamilton flows of metric functions, it is not obvious that the holomorphic configuration space
orbits  $\gamma_{x, \xi}^{\C}(t + i \tau)$  should be the same curves for fixed $\tau$ as the `phase space'
Hamilton orbits $\exp t \Xi_{\rho}$.  Thus, our purpose is to relate the following two families
of real curves:
\begin{itemize}

\item  Orbits of the Hamilton flow of $\sqrt{\rho}$ with respect to the \kahler form $\omega_{\rho}$
on the level set $\partial M_{\tau}$, i.e. with  $\sqrt{\rho} = \tau$;

\item The complex curve  $\gamma_{x, \xi}(t + i \tau)$  for fixed $\tau$. 

\end{itemize}

As a (simple) example we consider the flat torus $\R^n/\Z^n$. In this case $E(x, \xi) = x + i \xi$ and
$$\left\{ \begin{array}{l} \gamma_{x, \xi}(t) = x + t \frac{\xi}{|\xi|}, \;\;\; \gamma_{x, \xi}(t + i \tau) = x + (t + i \tau) \frac{\xi}{|\xi|},\;\;\\ \\
\mbox{for}\; \sqrt{\rho} (E(x, \xi) )=  |\xi| = \tau\;\;\; \exp t \Xi_{\sqrt{\rho}}(x + i \xi) = E(x + t \frac{\xi}{|\xi|},  \xi) = x +( t + i \tau) \frac{\xi}{|\xi|}
\end{array} \right. . $$

 Another way to contrast the two flows is the following:   $\partial M_{\tau}$ is a contact co-isotropic manifold for $\omega_{\rho}$ and the flow lines
of the Hamilton flow of $\sqrt{\rho}$ or of $\rho$ form  its real one dimensional  null-foliation. On the other hand, $dd^c \sqrt{\rho}$ has a complex one dimensional  null
foliation. We wish to relate the real null-foliation for $\omega_{\rho}$ on $\partial M_{\tau}$ to  the holomorphic
null foliation for  $dd^c \sqrt{\rho}$ on all of $M_{\tau}$.

\begin{prop}  \label{IMPID} The orbit of the Hamiltonian flow of $\sqrt{\rho}$ through $\gamma_{x, \xi}(i \tau)$ on $\partial M_{\tau}$
is the curve $t \to \gamma_{x, \xi}(t + i \tau)$.  
That is, $ \exp t \Xi_{\sqrt{\rho}} (E(x, \xi)) =  \gamma_{x, \xi}(t + i |\xi|)$. 
  \end{prop}
\begin{proof}  Let $(x, \xi) \in T^* M$ be a real co-vector.   It follows from Lemma
\ref{HAMID} that $\exp t \Xi_{\sqrt{\rho}}
(E(x, \xi)) = E \circ G^t(x, \xi)$ and then by Lemma \ref{pi} that
\begin{equation}\label{1} \pi G_{\C}^{t + i }(x, \xi)  = \pi G_{\C}^i G_{\C}^t(x, \xi) = E(G^t(x, \xi)) = \exp t \Xi_{\sqrt{\rho}} (E(x, \xi)). \end{equation}

Also,  we have (\cite{GLS}) 
\begin{equation} \label{EXPFORM} \exp_{\gamma(t)} i s \dot{\gamma}(t) = \gamma(t + i s ).  \end{equation}
Indeed, let  $\beta(r)$ be the geodesic with initial conditions $\beta(0), \dot{\beta}(0) = (\gamma(t), s \dot{\gamma}(t))$.
By definition, $\exp_{\gamma(t)} i s \dot{\gamma}(t)$ is the analytic continuation of $r \to \beta(r)$ at $r = i$. 
But $\beta(r) = \gamma(t + r s)$ for real $r$ so $\beta(i) = \gamma(t + i s)$. Therefore
\eqref{EXPFORM}  holds. 
Since $E(x, \xi) = \exp_x i \xi$ this says, 
$$E(\gamma(t),  s \dot{\gamma}(t)) = \gamma(t + i s). $$
Putting $s = |\xi|$ gives
\begin{equation} \label{MOREINT} E(G^t(x, \xi) ) = \gamma_{x, \xi}(t + i |\xi|). \end{equation}
Combining \eqref{1} and \eqref{MOREINT}  we have   \begin{equation} \exp t \Xi_{\sqrt{\rho}} (E(x, \xi)) = E \circ G^t(x, \xi) =  \gamma_{x, \xi}(t + i |\xi|), \end{equation}
completing the proof.

\end{proof}
%

The following Corollary is also in \cite{LS}.
\begin{cor}  $\gamma_{x, \xi}(i \tau)$ is a flow line of the gradient field $\nabla \sqrt{\rho}$.
\end{cor}





\subsection{\label{KISOMEM} \kahler isometric embedding of geodesics}

 As in the introduction, we let \eqref{GAMMAX}  be an arc-length
 parametrized geodesic, and let \eqref{gammaXCX}  denote its
analytic extension to a strip.  The special properties of geodesics are given in the following

\begin{prop} \label{ISOMETRY}  The map $\gamma_{x, \xi}: S_{\epsilon} \to M_{\epsilon}$ is 
a \kahler isometric embedding. More preicsely,
\begin{enumerate}

\item $\rho\gamma_{x, \xi}(t + i \tau) = \tau^2; $

\item $\sqrt{\rho}(\gamma_{x, \xi}(t + i \tau) = |\tau|; $ 

\item $\gamma_{x, \xi}^* \ddbar \rho = dt \wedge d \tau$;

\item  $\gamma_{x, \xi}^* \ddbar \sqrt{\rho} = \delta_0(\tau) dt d \tau. $

\end{enumerate} \end{prop}

This gives another  proof of Lemma \ref{IMPID} by showing that the Hamilton orbits  $\exp t \Xi_{\sqrt{\rho}}$ and  $\exp t \Xi_{\rho}$ are length minimizing curves
between their endpoints.   Indeed, by \cite{GLS} (p. 686), 
$$d^2 (q, \exp_q i \xi) = - |\xi|_q^2. $$
Similarly for $\rho(\exp_q i \xi)  = d^2(\exp_q i \xi, \exp_q - i \xi) $, which is determined from
analytic extension of $d^2(\exp_q \xi, \exp_q - \xi)$:

$$\rho (\gamma_{x, \xi}(t + i \tau) ) = d^2( \gamma_{G^t(x, \xi)} (i \tau),  \gamma_{G^t(x, \xi)} (-i \tau)) = 
- 4 \tau^2. $$

\begin{proof} Since $\gamma_{x, \xi}$ is holomorphic, the last two statements follow from the first
two, which are equivalent. The first statement follows from the fact that the holomorphic extension of
$\gamma_{x, \xi}$ corresponds under $E$ \eqref{EXP}  to the homogeneous lift of $\gamma_{x, \xi}$
to $T^* M$, i.e.
$$\tilde{\gamma}_{x,\xi}(t + i \tau) = \tau \gamma_{x, \xi}'(t) : S_{\epsilon} \to B^* M $$
and the holomorphicity of this map is essentially the definition of the adapted complex structure
in \cite{LS,GS1}.  Thus (as in Proposition \ref{IMPID}),  $\gamma_{x, \xi}(t + i \tau) = E(\gamma_{x, \xi}(t),  \tau \gamma_{x, \xi}'(t)) $.

\end{proof}


\section{\label{S} Szeg\"o kernels on boundaries of Grauert tubes}

The proofs of Proposition \ref{L2NORMintro} and Proposition \ref{LL} are based on microlocal analysis
in the complex domain.  In this section and the next, we introduce the basic objects of microlocal analysis
on Grauert tubes: the  \szego projectors and Poisson extension operator.  To the extent possible, we refer to \cite{BoGu,BoSj, GS2,Z} for background.
We include further background on Fourier integral operators with complex phase in the
Appendix \S \ref{APPENDIX}.

As in \cite{Z}, we study the analytic
continuation of eigenfunctions via the Poisson operator
\begin{equation} \label{POISSON} P^{\tau} : L^2(M) \to \ocal^0(\partial M_{\tau}) , \;\;\;
P^{\tau}(\zeta, y)  = \sum_j e^{- \tau \lambda_j} \phi_j^{\C}(\zeta) \phi_j(y), \end{equation}
which is defined by analytically continuing the Schwartz kernel of the Poisson semi-group
\begin{equation} \label{UITAU} U^{i \tau}= e^{- \tau \sqrt{\Delta}} : L^2(M) \to L^2(M)  \end{equation}
in the first variable to $\zeta \in M_{\tau}$ and then restricting it to $\partial M_{\tau}$. As reviewed in 
\S \ref{S} - \S \ref{POISSONKERNEL}, $\ocal^0(\partial M_{\tau})$ denotes the Hilbert space of boundary values
of holomorphic functions in $M_{\tau}$ which belong to $L^2(\partial M_{\tau})$.  The Poisson operator
is a Fourier integral operator with complex phase and its wave front is contained in the graph of the complexified
geodesic flow (see \S \ref{POISSONKERNEL}).

\subsection{ \label{SZEGO} \szego projector on boundaries of Grauert tubes}

We denote by $\ocal^{s + \frac{m-1}{4}}(\partial M_{\epsilon})$ the subspace of the Sobolev space $W^{s +
\frac{m-1}{4}}(\partial M_{\epsilon} )$ consisting of CR
holomorphic functions, i.e.
$${\mathcal O}^{s + \frac{m-1}{4}}(\partial M_{\epsilon} ) =
W^{s + \frac{m-1}{4}}(\partial M_{\epsilon} ) \cap \ocal (\partial
M_{\epsilon} ). $$ The inner product on $\ocal^0 (\partial
M_{\epsilon} )$ is with respect to the Liouville measure
$d\mu_{\epsilon}.$

For each $\tau < \epsilon$, the \szego projector
$$ \Pi_{\tau} : L^2(\partial M_{\tau}) \to
\ocal^0(
\partial M_{\tau} )$$  of the tube $M_{\tau}$ is the
 the orthogonal projection onto boundary values of
holomorphic functions in the tube. It is proved by Boutet de Monvel-Sj\"ostrand \cite{BoSj} that
 $\Pi_{\tau}$ is a complex Fourier integral
operator, whose real canonical relation is the graph
$\Delta_{\Sigma_{\tau}}$  of the identity map on the symplectic cone
$\Sigma_{\epsilon} \subset T^* (\partial M_{\tau} )$ spanned
by the contact form $\alpha_{\tau}: = d^c \rho|_{(\partial
M_{\tau} )}$, i.e.
$$\Sigma_{\tau} = \{(x, \xi; r \alpha_{\epsilon}),\;\;\;(x, \xi) \in \partial M_{\tau} ),\; r > 0 \}
 \subset T^* (\partial M_{\tau} ).\;\; $$
This symplectic cone is  symplectically equivalent to $T^* M$ under the homogeneous map 
\begin{equation} \label{ISOM} \iota_{r} : T^*M - 0 \to
\Sigma_{\tau},\;\; \iota_{\epsilon} (x, \xi) = (x, \tau
\frac{\xi}{|\xi|}, |\xi| \alpha_{(x, \tau \frac{\xi}{|\xi|})}
).
\end{equation} 
The parametrix construction of \cite{BoSj}  for Szeg\"o kernels of strictly pseudo-convex domains applies to $\Pi_{\tau}$
and has the form
\begin{equation} \label{PITAU} \Pi_{\tau}(x,y ) = \int_0^{\infty} e^{i \theta \psi(x, y)} A(x, y, \theta) d \theta, \end{equation}
where $\psi(x, y) = \tau - i r_{\C}(x, y) $ and where $A$ is a classical symbol of order $2m - 1$
in $\theta$. Here,  we use that $i r_{\C}(x, y)$ is the analytic extension of
$\sqrt{\rho}(x)$.

In the  theory
of Fourier integral operators with complex phase of \cite{MSj}, the full complex canonical relation of $\Pi_{\tau}$
lies  in the (Cartesian square of the) complexification $\tMpt$ of $\partial M_{\tau}$. 
In general, the  canonical relations associated to the Schwartz kernels $K(x, y)$ of Fourier
integral operators with complex phase lie in the complexification
of the relevant cotangent bundles $T^*(X \times X')$. But the wave front set $WF(K)$ is contained
in the real points of the canonical relation.   
Since the symplectic geometry
in the real domain is simpler,  we choose to work in the framework  of {\it adapted
Fourier integral operators}  and Toeplitz operators of   \cite{BoGu} (Section A.2 of the Appendix; see Definition 2.7).  
In the theory of Toeplitz operators of \cite{BoGu},
 a  special symbol  calculus is defined for a sub-class of   Fourier integral operators with  complex phase known
as  Toeplitz operators or Hermite Fourier integral operators   (see \S 3 of \cite{BoGu} for the definition). The  Toeplitz calculus of \cite{BoGu} applies to the operators  in our problem, and
it is not  necessary for us to analytically continue to  the  complexification $(\partial M_{\tau})^{\C}$.  
  But for the sake
of completeness we review the  full complex canonical relations 
and symbols in \S \ref{CCRPITAU}.
 For background on general complex canonical relations and their real points, we refer to  \cite{MSj, MS} (see also \S \ref{APPENDIX}).

We briefly recall the definition of adapted Fourier integral operators. 
Let  $X, X'$ be  two real $C^{\infty}$ symplectic manifolds and let $\Sigma \subset T^* X, \Sigma' \subset T^* X'$
be two symplectic subcones.  Let $\chi: \Sigma \to \Sigma'$ be a homogeneous symplectic isomorphism,
i.e one that commutes with the $\R_+$ action on the cones.  Then a Fourier 
integral operator with complex phase from $X$ to $X'$ is said to be {\it adapted} to $\chi$ if its complex canonical
relation $C$ is positive, and  if its real part  $C_{\R}$ is the graph of $\chi$. It is called elliptic  if its symbol is nowhere
vanishing.  
In this language, $\Pi_{\tau}$ is a Fourier integral operator with  complex phase which is adapted to
the identity map $\chi = Id_{\Sigma_{\tau}}: \Sigma_{\tau} \to \Sigma_{\tau}. $ 

As discussed in  the Appendix of \cite{BoGu} (section (2.11)),  a Fourier integral operator $A$ with complex phase adapted to $\chi$ is a Hermite operator. The authors
define the 
symbol at any point of the graph of $\chi$ to be a half-density on the graph of $\chi$ tensored with  a linear operator 
\begin{equation} \label{SYMBOL} K_{\sigma_A} f(u) = \lambda \int e^{- q(u, v)} f(v) dv  \end{equation}
on $\scal(\R^n)$ where $\lambda \in \C$ and $q(u,v)$ is a quadratic form with positive real part.  
Here, $\R^n$ is the symplectic orthogonal complement in $T (T^* \partial M_{\tau})$ to $T \Sigma$.  It
is shown in \cite{BoGu} that the principal symbol of any \szego projector $\Pi_{\tau}$  is a rank one projector
onto the `ground state'  annihilated by the Lagrangian system associated to the $\dbar$ operator.  By comparison,
in the symbol calculus of \cite{MSj}, the symbol is a half-density on the complex canonical relation. It is computed
in \cite{BoSj}.

\section{\label{POISSONKERNEL} Poisson kernels and their analytic continuations  }

In this section we study  the Poisson operator $P^{\tau}$  \eqref{POISSON} and its composition
with the wave group $U^t = e^{i t \sqrt{\Delta}}$.   We begin by reviewing
the proof of the following

\begin{prop}  \label{BGa} Let $(M, g)$ be a real analytic compact Riemanian $m$-manifold. Then $P^{\tau} U^t$ is a  
Fourier integral operator with complex phase of order $- \frac{m-1}{4}$  associated
to the positive complex canonical  relation
\begin{equation} \label{GAMMADEF}  \gcal_{\tau, t} = \{(y, \eta, \zeta, p_{\zeta} \} \subset T^*\tilde{M} \times T^* \tMpt:
(\zeta, p_{\zeta}) \sim G^{i t + t}(y, \eta)\}
\end{equation}
where $\sim$ denotes the equivalence relation in \S \ref{S}.

\end{prop}

Here, $\dim_{\C} \tilde{M} = m, \dim_{\C} \tMpt = 2 m -1$ and so the canonical relation
has complex dimension $3m - 1$. The graph $\{(y, \eta, G^{i \tau + t}(y, \eta)\}$ has complex  dimension $2 m$.

As mentioned above, it is simpler to work in the framework of adapted Fourier
integral operators, since the symbols live on the real points of the canonical relation. 
 In the following Proposition, we use an extension of the notion of adapted Fourier integral operator in the sense of
\S \ref{S}.  
Namely, we allow the homogeneous symplectic map to be a symplectic embedding rather than
a symplectic isomorphism. All of the composition results of \cite{BoGu} extend readily to this case.

\begin{prop}  \label{BG} $P^{\tau}$ is a  
Fourier integral operator with complex phase of order $- \frac{m-1}{4}$  adapted to 
the isomorphism $\iota_{\tau}: T^* M - 0 \to \Sigma_{\tau}$   \eqref{ISOM}. 
Moreover, for any $s$,
$$P^{\tau}: W^s(M) \to {\mathcal O}^{s +
\frac{m-1}{4}}(\partial M_{\tau} )$$ is a continuous
isomorphism.
\end{prop}

This Proposition is proved  in \cite{Z2, St}. A somewhat less precise statement  is given in \cite{Bou} and
in    Theorem
3.2 of \cite{GS2},
but the proof is hardly indicated there.  Since the statement and proof of Proposition \ref{BG}  differ somewhat 
from the previous versions, we review
its proof. We also need the following extension: 
\begin{prop}

 \label{REALPTAU}$P^{\tau} \circ U^t$ is a Fourier integral operator with complex phase  of order 
 $- \frac{m-1}{4}$ adapted to the symplectic isomorphism 
\begin{equation} \label{chitDEF} \chi_{\tau, t}  (y, \eta) = \iota_{\tau} (G^t(y, \eta), y, \eta ) : T^*M - 0 \to
\Sigma_{\tau}.
\end{equation}
Equivalently, $P^{\tau} \circ U^t$ is a   Fourier integral operator of Hermite type
of order 
 $- \frac{m-1}{4}$  associated
to the  canonical relation
\begin{equation} \label{GAMMAtDEF} \Gamma_{\tau, t} = \{(\iota_{\tau} (G^t(y, \eta), y, \eta ) \} \subset  \Sigma_{\tau}
\times T^* M. 
\end{equation}
\end{prop}

\begin{proof} This follows immediately from Proposition \ref{BG} and from the fact  proved in \cite{BoGu}, Theorems 3.4 and 7.5,  that the compositon of a Fourier integral operator
and a Fourier integral operator of Hermite type is also a Fourier integral operator of Hermite type, with a certain
addition law for the orders and a composition law for the symbols. \end{proof}

The proofs of Proposition \ref{BG}-\ref{REALPTAU} can be extracted from Hadamard's classical construction
of a convergent parametrix  for the Schwartz kernels of the the operators $\cos t \sqrt{\Delta}$ and
$\frac{\sin t \sqrt{\Delta}}{\sqrt{\Delta}}$ in a small neighborhood of the characteristic conoid. Hadamard
did not consider the operator   $\exp (i t \sqrt{\Delta})$ since $\sqrt{\Delta}$ was unknown at that time, but
in \cite{Z2} we modify the Hadamard parametrix construction
to construct the Schwartz kernel $U(t, x, y)$  of $U^t$ 
as a convergent Riesz series expansion near the characteristic conoid. Further it is shown that 
$U(i \tau, x, y)$ admits an analytic continuation in $x$ when $(i\tau, x, y)$ lies in a small enough neighborhood
of the complex characteristic conoid. Outside of this neighborhood $U(i \tau, x, y)$ is real analytic. 
It follows from the construction  that  $U(i \tau)$ is a Fourier integral operator of complex type, and its canonical relation
is the  graph of the imaginary time geodesic flow $G^{i \tau}$; we refer to \cite{Z2} for the details.

However,  we need   one additional statement to justify the claim that $P^{\tau}$ is adapted to $\iota_{\tau}$. The canonical
relation $\gcal_{\tau, t}$  of Proposition \ref{BG} is the holomorphic extension of $G^t$.  We need to see that the intersection
of this graph with $\Sigma_{\tau} \times T^* M$ is  $\Gamma_{\tau, t}. $  The proof involves the same identifications
as in \S \ref{HAMFLOWREL}. 

We recall that $G^t: T^*M - 0 \to T^*M - 0$ is the homogeneous geodesic flow, i.e.
$G^t(x, \xi) = |\xi| G^t(x, \frac{\xi}{|\xi|})$. The analytic continuation in $t$ is also homogeneous, so we have
\begin{equation} \label{HOM} G^{i \tau} (x, \xi) = |\xi| G^t{i \tau} (x, \frac{\xi}{|\xi|}). \end{equation}

\begin{lem}   \label{GTAU} We have: $\iota_{\tau}(y, \eta) = G^{i \tau}(y, \eta). $ Thus, $G^{i \tau}$ gives a homogeneous
symplectic isomorphism
$$G^{i \tau}: T^* M - 0  \to \Sigma_{\tau} - 0. $$
\end{lem}

\begin{proof}  
By homogeneity we may assume 
that $|\xi| = 1$. We then want to show that
$$G^{i \tau} (x, \xi) =  (\exp_x i \tau \xi,  \alpha_{\tau}) $$ where $\alpha_{\tau}$ is the canonical
one form of $T_{\exp_x i \tau \xi}^* \partial M_{\tau}$. Note that  $G^t(x, \xi) = g^t(x, \xi)$ for $|\xi| = 1$
 and we can analytically
continue $g^t(x, \xi)$ in $t$.
By definition, for $(x, \xi) \in S^* M \subset T^*M_{\C}$, 
$$g^{i \tau}(x, \xi) \in T^*_{ \gamma_{x,\xi}(i \tau)  } M_{\tau} $$
is defined by extending the phase space  Hamilton orbit $\tilde{\gamma}_{x, \xi}: \R \to T^* M$
 holomorphically in time as an orbit of flow of the holomorphic Hamiltonian
\eqref{HC} and for fixed $(x, \xi)$ is is   a holomorphic
strip in $T^* M_{\C}$.  
We then restrict it to the imaginary time axis   
 to obtain a real curve
\begin{equation} \tilde{\gamma}_{x, \xi} (i \tau): [0, \epsilon_0) \to T^* M_{\C}. \end{equation}
If $\pi_{\C}: T^* M_{\C} \to M_{\C}$ denotes the natural projection, then
$ \pi_{\C} (\tilde{\gamma}_{x, \xi}(i \tau)) = \gamma_{x, \xi} (i \tau)$, so     the only  non-obvious aspect is
the $p_{\zeta}$ component. But  $\tilde{\gamma}_{x, \xi}(i \tau)$ is the cotangent
lift of $ \gamma_{x, \xi} (i \tau)$, hence is given by 
\begin{equation}\label{LIFT} G^{i \tau}(x, \xi) = ( \gamma_{x, \xi} (i \tau), g_{\C} \frac{d}{d \tau}   \gamma_{x, \xi} (i \tau)) \end{equation}
where
 $g_{\C}: T M_{\C} \to T^* M_{\C}$ is  the linear map defined by the analytic continuation \eqref{gC}  of the Riemannian metric to a holomorphic symmetric two tensor on $T M_{\C}$. 
The formula \eqref{LIFT} is the analytic continuation of the lifting formula in the real domain. It remains to prove that
\begin{equation} \label{CLAIM} ( \gamma_{x, \xi} (i \tau), g_{\C} \frac{d}{d \tau}   \gamma_{x, \xi} (i \tau)) 
= \alpha_{  \gamma_{x, \xi} (i \tau)) }. \end{equation}
We recall that
$\frac{d}{d \tau}  \gamma_{x, \xi} (i \tau)  \in T_{\gamma_{x, \xi}(i \tau)} M_{\C}$ and that $\gamma_{x, \xi}(i \tau) \in
\partial M_{\tau}$. The curve $t \to \gamma_{x, \xi}(t + i \tau)$ is characteristic for the form  $\omega$ restricted
to $\partial M_{\tau}$, and so $\frac{d}{dt} \gamma_{x, \xi}(t + i \tau)$ spans the null space of the form. Note that
$\frac{d}{d \tau} \gamma_{x, \xi}(i \tau) = i \frac{d}{dt} \gamma_{x, \xi}(t + i \tau) |_{t = 0}$ and
that  $T := \frac{d}{dt} \gamma_{x, \xi}(t + i \tau)|_{t = 0}$ and $ \frac{d}{d \tau} \gamma_{x, \xi}(i \tau) 
= J T$ are symplectically paired by  $\omega_{\gamma_{x, \xi}(i \tau)}.$

 We now dualize to $T^* \partial M_{\tau}$  using the metric. We let $\alpha_{\tau}$
denote the canonical 1-form of $T^* \partial M_{\tau}$. Then $\alpha_{\tau} (T) = 1$
and $\ker \alpha_{\tau}$ is the CR sub-bundle of $T \partial M_{\tau}$. 
 The claim \eqref{CLAIM} is equivalent to
\begin{equation} \label{CLAIMb} g_{\C}(\frac{d}{d \tau} \gamma_{x, \xi} (i \tau),  \cdot)  =
\omega(J T, \cdot) = 
\alpha_{\gamma_{x, \xi}(i \tau)} (\cdot) \in T^*_{\gamma_{x, \xi}(i \tau)} M_{\tau}. \end{equation}
Indeed, we note that we get $1$ if we evaluate both sides on $T$, and so it suffices to
check that we get $0$ on the CR sub-bundle. But the CR sub-bundle is $J$-invariant
and symplectically orthogonal to $T$, and therefore also to $J T$. 

Alternatively, we compute using   Hamilton's equations \eqref{XI}. 
If $G^{i \tau}(x, \xi) = (\zeta, p)$, we have
$$\begin{array}{lll} g_{\C} \frac{d}{d \tau} \gamma_{x,\xi}(i \tau)  &= & g_{\C} (\sum_{i, j} g^{ij}( \gamma_{x, \xi}(i \tau))
p_i \frac{\partial}{\partial \zeta_j} \\ &&\\
& = &  (\sum_{i, j, k} g^{ij}( \gamma_{x, \xi}(i \tau))
p_i  g_{j k}( \gamma_{x, \xi}(i \tau)) d \zeta_k \\ &&\\
& = & \sum_k  p_k d \zeta_k = \alpha_{( \gamma_{x, \xi}(i \tau))}. \end{array} $$
Here, we also use that
$\alpha_{\C} |_{T \Sigma} = \alpha_{\Sigma}. $
In other words, if we restrict the holomorphic canonical one form to $T \Sigma$, then it agrees with the
restriction of the canonical one form of $T^*(\partial M_{\tau})$ to $T \Sigma$. But this is obvious since
$\alpha_{\C}$ is the analytic continuation of the canonical 1-form of $T^* \partial M_{\tau}$.



\end{proof}


\subsection{\label{ATAUsect} Poisson kernel and \szego kernel}

Define the self-adjoint operator on $L^2(M, dV)$ by 
\begin{equation} \label{ATAU} A_{\tau} = (P^{\tau *} P^{\tau})^{-\half}. \end{equation}
By Lemma 3.1 of  \cite{Z}  $A_{\tau}$   is an elliptic  self-adjoint  pseudo-differential operator of order $ \frac{m-1}{4}$
with principal symbol $|\xi|^{ \frac{m-1}{4}}$. That is, $(P^{\tau *} P^{\tau}$ is a pseudo-differential
operator of order $-\frac{m-1}{2}$ with principal symbol $|\xi|^{- \frac{m-1}{2}}$.

\begin{prop} \label{PS} The \szego kernel $\Pi_{\tau}$ is related to the Poisson kernel $P^{\tau}$ by 
$$\Pi_{\tau} = P^{\tau} A_{\tau}^2  P^{\tau *}. $$
\end{prop}

\begin{proof} Let  $\tilde{\Pi}_{\tau}$ denote
the operator in the statement.  It is easily verified that 
$$\tilde{\Pi}_{\tau}^* = \tilde{\Pi}_{\tau}, \;\;\;\;\; \tilde{\Pi}_{\tau} \tilde{\Pi}_{\tau} = \tilde{\Pi}_{\tau}. $$
To check the idempotent property we note that
$$ P^{\tau} \left(P^{\tau *} P^{\tau} \right)^{-1} P^{\tau *}  P^{\tau} \left(P^{\tau *} P^{\tau} \right)^{-1} P^{\tau *}
=  P^{\tau} \left(P^{\tau *} P^{\tau} \right)^{-1} P^{\tau *}.$$

Further the range of $\tilde{\Pi}_{\tau}$ equals $\ocal^0(\partial M_{\tau})$. Thus it is a projection
onto that space. Moreover it is the orthogonal projection since it annihlates any $f \bot \ocal^0(\partial
M_{\tau})$: Indeed,  $\ocal^0(\partial M_{\tau})$  is the
image  of $P^{\tau}$  on $L^2(M)$ and 
$P_{\tau}^* f = 0$ if $f \bot \ocal^0(\partial M_{\tau})$. 

\end{proof}

\begin{prop} \label{UNIT} The operator $ P^{\tau} A_{\tau}$ is a unitary operator: $L^2(M) \to \ocal^0(\partial M_{\tau})$
with left inverse $A_{\tau} P^{\tau *}  $.  Moreover,    $ A_{\tau}^2 P^{\tau *} : \ocal^0(\partial M_{\tau}) \to L^2(M)$ is a left inverse to
$P_{\tau}$. \end{prop}

\begin{proof} It suffices to show that
$$\left\{ \begin{array}{l} A_{\tau} P^{\tau *} P^{\tau } A_{\tau} = A_{\tau}^2 P^{\tau *} P^{\tau}= Id: L^2(M) \to L^2(M),\\ \\
P^{\tau} A_{\tau}^2 P^{\tau *}  = \Pi_{\tau}: \ocal^0(\partial M_{\tau}) \to \ocal(\partial M_{\tau}) \end{array} \right. $$
The identities are obvious (see Proposition \ref{PS}).
\end{proof}

We note that the CR functions $\{u_j^{\tau}\}$ are not orthogonal. However, it follows from Proposition \ref{UNIT} that
\begin{cor} \label{ONB}  The CR functions $P^{\tau} A_{\tau} \phi_j$ form an orthonormal basis of $H^2(\partial M_{\tau}, d\mu_{\tau})$.
\end{cor}
Since $A_{\tau} = \Delta^{\frac{m-1}{8}}  \; \mbox{mod}\;\; \Psi^{\frac{m-1}{4} - 1}$, the basis
$\{u_j^{\tau}\}$ is almost  orthonormal, i.e. the Gram matrix $\left( \langle u_j^{\tau}, u_k^{\tau} \rangle \right)
= I + K_{\tau}$ where $K_{\tau}$ is a compact operator.

\subsection{\label{LWLSECT} Pointwise Weyl laws and norms for complexifed eigenfunctions}

We denote by $$\Pi_{I_{\lambda}}(x, y)  = \sum_{j: \lambda_j \in  I_{\lambda}} \phi_j(x) \phi_j(y) $$
 the spectral projections for $\sqrt{\Delta}$ corresponding to the interval
$I_{\lambda}$, which we take to be either $[0, \lambda]$ or $[\lambda, \lambda + 1].$
The analytic continuations of the spectral projection kernels are then 
\begin{equation}\label{CXSPMa}   \Pi_{I_{\lambda}}^{\C}(\zeta, \bar{\zeta}) =
\sum_{j: \lambda_j \in  I_{\lambda}}
|\phi_{\lambda_j}^{\C}(\zeta)|^2.
\end{equation}  Since they are of  exponential growth, we also define  their  `tempered' analogues,
\begin{equation}\label{TCXSPM}   P_{ I_{\lambda}}^{\tau}(\zeta, \bar{\zeta}) =
\sum_{j: \lambda_j \in  I_{\lambda}} e^{- 2 \tau \lambda_j}
|\phi_{\lambda_j}^{\C}(\zeta)|^2, \;\; (\sqrt{\rho}(\zeta) \leq
\tau).
\end{equation}

We then have a  one-term pointwise l Weyl law for complexified
spectral projections:
\begin{prop}\label{PTAULWL}  \cite{Z2}  On any compact real analytic Riemannian manifold $(M,g)$ of dimension $m$,  we have,  with remainders uniform in
$\zeta$, For $\sqrt{\rho}(\zeta) \geq \frac{C}{\lambda}, $
$$ P_{[0, \lambda]}(\zeta, \bar{\zeta}) =  (2\pi)^{-n} \left(\frac{\lambda}{\sqrt{\rho}} \right)^{\frac{n-1}{2}}
  \left( \frac{\lambda}{(m-1)/2 + 1} +  O (1) \right);
$$

\end{prop}

Although we do not need it here, the asymptotics 
for $\sqrt{\rho}(\zeta) \leq \frac{C}{\lambda}$ are given by  $$ P_{[0,
\lambda]}(\zeta, \bar{\zeta}) =  (2\pi)^{-m} \; \lambda^{m}
\left(1 + O(\lambda^{-1}) \right]. $$
We now set $\sqrt{\rho} = \tau$ and define
$$M(\lambda) =  (2\pi)^{-m}  \tau^{- \frac{m-1}{2}}  \frac{\lambda^{\frac{m+1}{2}}}{(m-1)/2 + 1}. $$

The pointwise Weyl laws  imply upper bounds on sup norms of complexified eigenfunctions \cite{Bou,GLS,Z2}.

\begin{prop} \label{PWa} \cite{Z2} Suppose  $(M, g)$ is real analytic of dimension $m$. Then,

\begin{enumerate}

\item For $\tau \geq \frac{C}{\lambda} $ and $\sqrt{\rho}(\zeta) = \tau$, 
there exists $C > 0$ so that
$$ C
\lambda_j^{-\frac{m-1}{2}} e^{ \tau \lambda} \leq \sup_{\zeta \in
M_{\tau}} |\phi^{\C}_{\lambda}(\zeta)| \leq C
  \lambda^{\frac{m-1}{4} + \half} e^{\tau \lambda}.
$$

\item  For $\tau\leq \frac{C}{\lambda}, $ and $\sqrt{\rho}(\zeta) = \tau$,  there
exists $C > 0$ so that
$$ |\phi^{\C}_{\lambda}(\zeta)| \leq \lambda^{\frac{m - 1}{2}};
$$

\end{enumerate}

\end{prop}

By  Corollary \ref{ONB}  the operator $P^{\tau} A^{\tau}$ is unitary so that
$\{P^{\tau} A_{\tau} \phi_j\}$ are orthonormal.  They are asymptotically the same as
$$\lambda_j^{\frac{m-1}{4}} P^{\tau} \phi_j = \lambda_j^{\frac{m-1}{4}} e^{- \tau \lambda_j} \phi_j^{\C} |_{\partial M_{\tau}}, $$hence these are almost normalized to have $L^2$-norms equal
to one.

\section{\label{WG} The wave group in the complex domain }

To prove Proposition \ref{LL}, we need to consider the symmetries of the matrix elements. 
The key operator in studying restriction to complexified geodesics is the composition
\begin{equation} \gamma_{x, \xi}^{\tau *} \circ P^{\tau}: L^2(M) \to \ocal(\partial S_{\tau}), \end{equation}
where $\gamma_{x, \xi}^{\tau *}$ is the pullback under \eqref{gammatau}.  To prove Proposition \ref{LL}, 
we show that to leading order, this operator intertwines the wave group $U^t = e^{i t \sqrt{\Delta}}$ on $M$ with the
translation operator $T_t f(s + i \tau) = f(s + t + i \tau)$ on $\partial S_{\tau}$.  On the infinitesimal
level, $ \gamma_{x, \xi}^{\tau *} \circ P^{\tau}$ intertwines $\frac{\partial}{\partial t}$ with the generator
$\Xi_{\sqrt{H}}$ of the the geodesic flow.  Here and below, we identify a vector field $\Xi$ with the differential
operator $\Xi(f) = df (\Xi)$. 

It is simpler 
to work with a unitarily equivalent ``wave group"  $\tilde{V}_{\tau}^t$ acting on the $\ocal^0(\partial M_{\tau})$.
In this section, we define and study this wave group. The symbol calculus of  Toeplitz Fourier integral operators  shows that it is essentially the compression
to $\ocal^0(\partial M_{\tau})$ of translation by the (non-holomorphic) Hamilton flow of $\sqrt{\rho}$. This
makes it evident that pullback by $\gamma_{x, \xi}^{\tau}$ intertwines the wave group with translations.

\subsection{The wave group $\tilde{V}_{\tau}^t$} .
\bigskip

We first define the operators $V_{\tau}^t$:

\begin{defin} \label{VINTRO}  Let $A_{\tau} = (P^{\tau *} P^{\tau})^{-\half}$ \eqref{ATAU}. Then set
$$V_{\tau}^t = P^{\tau} U^t A_{\tau}^2 P^{\tau *}: \ocal^0(\partial M_{\tau}) \to \ocal^0(\partial M_{\tau}). $$
\end{defin}

The advantage of $V_{\tau}^t$ is that its eigenfunctions are natural:
\begin{equation}  \label{VTEV}  V_{\tau}^t  P^{\tau} \phi_j = e^{i t \lambda_j} P^{\tau} \phi_j. \end{equation}
Indeed,  from Proposition \ref{UNIT} ,  $ A_{\tau}^2 P^{\tau *}$ is a left inverse to $P^{\tau}$. Hence it suffices to observe that
$$ P^{\tau} U^t A_{\tau}^2 P^{\tau *} P^{\tau} \phi_j = e^{i t \lambda_j} P^{\tau} \phi_j. $$

The disadvantage is
that $V_{\tau}^t$ is not a unitary group
and is not even normal since the $A_{\tau}^2$ factors moves to the left side of $U^t$
when taking the adjoint and in general $U^t$ and $A_{\tau}$ do not commute.  So the spectral theorem does
not apply to $V_{\tau}^t$.  This defect is remedied by:
\begin{defin} \label{VINTROalt}   Define 
$$\tilde{V}_{\tau}^t = P^{\tau} A_{\tau} U^t A_{\tau} P^{\tau *}: \ocal^0(\partial M_{\tau}) \to \ocal^0(\partial M_{\tau}). $$
\end{defin}

\begin{prop} \label{TILDEV}  $\tilde{V}_{\tau}^t$ is a unitary group   with eigenfunctions
$$\tilde{V}_{\tau}^t P^{\tau} A_{\tau} \phi_j = e^{i t \lambda_j} P^{\tau} A_{\tau} \phi_j . $$\end{prop}

\begin{proof}  By Proposition \ref{UNIT}, 
$$\begin{array}{lll} \tilde{V}_{\tau}^t \tilde{V}_{\tau}^{t *} & = &  P^{\tau} A_{\tau} U^t A_{\tau} P^{\tau *}  P^{\tau} A_{\tau} U^{-t} A_{\tau} P^{\tau *} \\ &&\\ && 
= P^{\tau } P^{\tau *} = \Pi_{\tau}: L^2(\partial M_{\tau}) \to L^2(\partial M_{\tau}), \end{array}$$
so that $\tilde{V}_{\tau}^t$ is unitary. Also, 
$$\begin{array}{lll} \tilde{V}_{\tau}^t \tilde{V}_{\tau}^{s} & = &  P^{\tau} A_{\tau} U^t A_{\tau} P^{\tau *}  P^{\tau} A_{\tau} U^{s} A_{\tau} P^{\tau *} \\ &&\\ && 
= P^{\tau } A_{\tau} U^{t+s}A_{\tau}P^{\tau *}  =  \tilde{V}_{\tau}^{t + s} : L^2(\partial M_{\tau}) \to L^2(\partial M_{\tau}), \end{array}$$

Similarly,
$$\tilde{V}_{\tau}^t P^{\tau} A_{\tau} \phi_j =   P^{\tau} A_{\tau} U^t A_{\tau} P^{\tau *}  P^{\tau} A_{\tau} \phi_j 
= e^{i t \lambda_j} P^{\tau} A_{\tau} \phi_j. $$

\end{proof}

\subsection{\label{VTOEP} $V_{\tau}^t$ as a Fourier integral Toeplitz operator}


The following Proposition states the analogue of Propositions \ref{BGa}-\ref{REALPTAU} for $V_{\tau}^t$. 
It shows that  $V^t_{\tau}$ is a Fourier integral operator with complex phase of Hermite type on $
\ocal^0(\partial M_{\tau}) \subset L^2(\partial M_{\tau})$ 
which  is ``adapted" to the graph of the Hamiltonian flow of $\sqrt{\rho}$ on the symplectic cone $\Sigma_{\tau} $
associated to the Hardy space $\ocal^0(\partial M_{\tau})$ in the sense of \cite{BoGu} (see \S \ref{S}).  

\begin{prop}  \label{VWAVE} $V^t_{\tau}$ and $\tilde{V}_{\tau}^t$  are Fourier integral operators with complex phase
 of Hermite type on $
\ocal^0(\partial M_{\tau}) \subset L^2(\partial M_{\tau})$ 
adapted to the graph of the Hamiltonian flow of $\sqrt{\rho}$ on $\Sigma_{\tau}. $  They have the 
same canonical relations and principal symbols
$$ \sigma_{V_{\tau}^t} = |\xi|^{\frac{m-1}{4}} \sigma_{P^{\tau}} \circ \sigma_{U^t} \circ \sigma_{P^{\tau *}}. $$

 \end{prop}

\begin{proof}  That both operators have the same canonical relation and symbol is obvious because they
only differ in the order of the pseudo-differential operator $A_{\tau}$, which has a scalar symbol. 
By definition   and by Proposition \ref{BGa}, $\tilde{V}_{\tau}^t$ is a composition of Fourier integral operators with complex phase, and 
all are associated to canonical graphs and equivalence relations. Moreover all are operators
of Hermite type.  If follows that the composition is transversal,
so that $\tilde{V}_{\tau}^t$  is also a  Fourier integral operator with complex phase and of Hermite type. 

The underlying complex canonical
transformation is $$ \begin{array}{l}   \gcal_{\tau, t}  \circ \gcal_{\tau}^*  
= \{(\zeta, p_{\zeta} , \zeta', p_{\zeta'} \} \subset  T^* \tMpt  \times T^* \tMpt: \\ \\
\exists (y, \eta) \in T^* M: 
(\zeta, p_{\zeta}) \sim G^{i \tau + t}(y, \eta) , \;\;\;
(\zeta', p_{\zeta'}) \sim G^{i \tau}(y, \eta) \}\end{array}$$ where $ \gcal_{\tau, t} $ is defined  in \eqref{GAMMADEF}. 
Thus 
$$
(\zeta, p_{\zeta}) \sim G^{t} (\zeta', p_{\zeta}')$$
and there is the additional constraint that $(\zeta', p_{\zeta'}) \sim G^{i \tau}(y, \eta)  \in \Sigma_{\tau}$. 

The real points of this canonical relation is its intersection with $\Sigma_{\tau} \times \Sigma_{\tau}$,
and then we have 
$$ \begin{array}{l}   \gcal_{\tau, t}  \circ \gcal_{\tau}^*  \cap \Sigma_{\tau} \times \Sigma_{\tau} 
= \{(\zeta, p_{\zeta} , \zeta', p_{\zeta'} \} \subset  \Sigma_{\tau} \times \Sigma_{\tau}: 
(\zeta, p_{\zeta})  =G^{ t}
(\zeta', p_{\zeta'})  \}\end{array}.$$ 
But then $\zeta' \in \partial M_{\tau}, p_{\zeta} = c d^c \sqrt{\rho}$ and by Proposition \ref{IMPID} and Lemma
\ref{GTAU},
$$G^t(\zeta',  d^c \sqrt{\rho}) = (\exp t \Xi_{\sqrt{\rho}}(\zeta'),    d^c \sqrt{\rho}). $$
It follows that $V_{\tau}^t$ is adapted to the graph of  $E  G^t E ^{-1}= \exp t \Xi_{\sqrt{\rho}}$
on  $\Sigma_{\tau}$.

The symbol calculation is then a routine use of the composition
 theory of Fourier integral operators and therefore we only outline it here: 
The symbol of $U^t$ is well-known to be the canonical volume half-density on the graph of $G^t$.  On the real
points of the canonical relation $C_{\tau}$ of $\Pi_{\tau}$, this volume half-density is conjugated to the 
canonical volume half-density on the graph of $\exp t \Xi_{\rho} $ in $\Sigma_{\tau} \times \Sigma_{\tau}$.

\end{proof}



\subsection{$V^t_{\tau}$ as a Toeplitz Fourier integral operator}

We recall that in \cite{BoGu} a special symbol calculus was developed for Fourier integral operators
of Hermite type adapted to symplectic diffeomorphisms, i.e. for compressions (restrictions)
 $\Pi_{\tau} F \Pi_{\tau}$ 
of Fourier integral operators $F$  of complex type to the Hardy space $H^2(\partial M_{\tau}$. Since
the operators $\tilde{V}_{\tau}^t$ and $V_{\tau}^t$ of this type, they have Toeplitz symbols in this
sense. In the next Proposition, we calculae these symbols and put the operators
into a more geometric form of Toeplitz quantizations of a Hamiltonian flow. Such a Toeplitz quantization 
was defined by Boutet de Monvel \cite{Bou2} and more systematically developed in \cite{Z3}.

To state the next Proposition, we need some further notation and background on symbols of Szeg\"o
projectors. We recall from \cite{BoGu}  that $\sigma_{\Pi_{\tau}}$ is a rank one projector onto a ground state $e_{\Lambda}$
in the quantization of the symplectic transversal $T \Sigma^{\perp} \subset T^* \partial M_{\tau}$
to $T \Sigma$. The ground state is annihilated by a Lagrangian system of  Cauchy-Riemann equations corresponding
to a Lagrangian subspace $\Lambda \subset T \Sigma$.

The time evolution of $\Pi_{\tau}$ under the flow $g^t_{\tau}$ is defined by
\begin{equation} \label{PITAUT} \Pi_{\tau}^t = g_{\tau}^{-t}  \Pi_{\tau}  g_{\tau}^t. \end{equation}   
It is another \szego projector adapted  to the symplectic cone $\Sigma_{\tau}$;
since $g^t$
is not a family of holomorphic maps in general, $\Pi_{\tau}^t$ is
associated to a new CR (complex) structure and  translation by $g^t_{\tau}$   does not commute with $\Pi_{\tau}$. 
But $\Sigma_{\tau}$ is invariant under the flow and $g^t$ clearly commutes with the identity map on $\Sigma_{\tau}$.
The change in the range of $\Pi_{\tau}^t$ under $t$ is encoded to leading order in the change of the principal symbol.

 Under $dg^t_{\tau}$, the Lagrangian $\Lambda$  goes to a new
Lagrangian $\Lambda_t$ and $\sigma_{\Pi_{\tau} ^t}$ is 
a rank one projector onto a ground state $e_{\Lambda_{\tau}^t}$  depending on $t$.
 If we right compose
by $\Pi_{\tau}$ the composite symbol is  \begin{equation}
\label{SYMBPIT}  \sigma(\Pi_{\tau} \Pi_{\tau}^t \Pi_{\tau}) = | \langle e_{\Lambda_{\tau}} , e_{\Lambda_{\tau}^t} \rangle |^2
\sigma_{\Pi_{\tau}}. \end{equation}
The overlap $\langle e_{\Lambda_{\tau}} , e_{\Lambda_{\tau}^t} \rangle $ is an inner product of two Gaussians
and is calculated explicitly in \cite{Z3}.  It is nowhere vanishing.

\begin{prop}\label{VT} Let $g^t_{\tau} = \exp t \Xi_{\sqrt{\rho}}$ on $\partial M_{\tau}$ and let
\begin{equation} \label{sigmataut} \sigma_{\tau, t} =  \langle e_{\Lambda_t}, e_{\Lambda} \rangle^{-1}. \end{equation}
 Then
\begin{equation} \label{CXWAVEGROUP} V^t_{\tau}: = \Pi_{\tau} g_{\tau}^t \sigma_{t, \tau} \Pi_{\tau},
\;\;\;\; \tilde{V}_{\tau}^t  = \Pi_{\tau} \sqrt{\sigma_{t, \tau}} g_{\tau}^t \sqrt{\sigma_{t, \tau}} \Pi_{\tau} 
\end{equation} \end{prop}

\begin{proof}  It suffices to observe that each side of each formula is an elliptic  Toeplitz Fourier integral operator
associated to the graph of $g_{\tau}^t$.  In the case of $V_{\tau}^t$ and $\tilde{V}_{\tau}^t$ this follows
from Proposition \ref{REALPTAU} and by the composition theorem for such operators in \cite{BoGu}. In the 
case of $ \Pi_{\tau} g_{\tau}^t \sigma_{t, \tau} \Pi_{\tau}$ it follows similarly from the fact that $\Pi_{\tau}$
is a Toeplitz operator and from the simple composition with pullback by $g_{\tau}^t$ (see \S \ref{SZEGO}). 

 To compute the symbols we use Proposition \ref{TILDEV}.  On the principal symbol level it implies that
$$\sigma_{\Pi_{\tau}} \circ \sigma_{\tau, -t} \sigma_{g^{-t} \Pi_{\tau} g^t} \sigma_{t, \tau}
\circ \sigma_{\Pi_{\tau}} = \sigma_{\Pi_{\tau}} \leftrightarrow  |\sigma_{\tau, t} |^2 
\sigma_{\Pi_{\tau}} \circ \sigma_{g^{-t} \Pi_{\tau} g^t} 
\circ \sigma_{\Pi_{\tau}} = \sigma_{\Pi_{\tau}} .$$

Then
$$\sigma_{\Pi_{\tau}} \circ \sigma_{g^{-t} \Pi_{\tau} g^t} 
\circ \sigma_{\Pi_{\tau}}  = |\langle e_{\Lambda_t}, e_{\Lambda} \rangle|^2. $$
It follows that 
$$ |\sigma_{\tau, t} |^2  =  |\langle e_{\Lambda_t}, e_{\Lambda} \rangle|^{-2}. $$
There is no unique solution $\sigma_{\tau, t}$ but we can choose

 \end{proof}


\section{\label{INT} Restriction to geodesics I: Intertwining relations}

The purpose of this section is to prove that the restriction $\gamma_{x, \xi}^{\tau*}$ intertwines translation
on $\R$ and translation by the geodesic flow on $\partial M_{\tau}$. There are several further  intertwining relations of
this kind, both infinitesimal and global.  The one we need for the proof of Proposition \ref{LL} is a $T^* T$ version on $\partial M_{\tau}$,
and so it is emphasized in the following:

\begin{mainprop}  \label{INTERTWINING2}  On $L^2(\partial M_{\tau})$, we have

$$V_{\tau}^{*t} (\gamma_{x, \xi}^{\tau})^{**}  Op_{\lambda_j}(a) (\gamma_{x, \xi}^{\tau})^*  V_{\tau}^t
\simeq \Pi_{\tau}  ((\gamma_{x, \xi}^{\tau})^* )^* T_t^* Op_{\lambda_j}(a)  T_t (\gamma_{x, \xi}^{\tau})^* \Pi_{\tau},$$
where $\simeq$ means that both sides belong to the same class of Fourier integral operators and have the same
principal symbols.  
.

\end{mainprop}

By adding lower order terms to the symbols we can arrange the left side to equal the right modulo Toeplitz
smoothing operators.  We begin by using Proposition \ref{VT} to show that
Proposition \ref{INTERTWINING2} is equivalent to
\begin{lem} \label{EQUIV}  We have,
$$\Pi_{\tau} ((\gamma_{x, \xi}^{\tau})^*)^*  T^{t*} Op_{\lambda}(a) T^t (\gamma_{x, \xi}^{\tau})^* \Pi_{\tau}  \simeq
\Pi_{\tau} \overline{\sigma_{\tau t} }g^{-t}  \Pi_{\tau} 
(\gamma_{x, \xi}^{\tau})^{**}  Op_{\lambda_j}(a) (\gamma_{x, \xi}^{\tau})^*
\Pi_{\tau} \sigma_{\tau t} g^{t}  \Pi_{\tau} ,$$
in the sense that the operators on each side are Hermite Fourier integral operators (Toeplitz operators) adapted
to the symplectic  embedding $\iota_{x, \xi}^{\tau}: \R_+ \gamma_{x, \xi}^{\tau} \subset \Sigma_{\tau}$  and having
the same principal symbol.
\end{lem}

Indeed the equivalence follows immediately from Proposition \ref{VT} and from the fact that
 \begin{equation}  \label{SYMEMBED} (\gamma_{x, \xi}^{\tau})^* \Pi_{\tau} 
: \ocal^0(\partial M_{\tau}) \to \ocal^{-1}_{loc}(\partial S_{\tau})  \end{equation} is the adjoint of a Toeplitz
Hermite Fourier integral operator  adapted to the symplectic embedding
 $\R_+ \gamma_{x, \xi}^{\tau} \subset \Sigma_{\tau}$. The latter statement is proved in 
Theorem 9.1 of \cite{BoGu}. All of the hypotheses of  that theorem  are satisfied by the map
$$\gamma_{x, \xi}^{\tau} \times \gamma_{x, \xi}^{\tau}: \R \times \R \to \partial M_{\tau} \times \partial M_{\tau}.$$
It follows that 
\begin{equation} \label{T*T} (\gamma_{x, \xi}^{\tau})^* \Pi_{\tau}
(\gamma_{x, \xi}^{\tau})^{**}: \ocal^0(\partial M_{\tau}) \to \ocal^{0}(\partial M_{\tau}) \end{equation}
 is a Toeplitz
operator  quantizing the symplectic sub-cone
 $\R_+ \gamma_{x, \xi}^{\tau}  \subset \Sigma_{\tau}$, i.e. having the real points of its canonical relation
along $\Delta_{\R_+ \gamma_{x, \xi}^{\tau} \times \R_+ \gamma_{x, \xi}^{\tau}}.$ It has a paramatrix
of the form,
$$\Pi_{\tau}(\gamma_{x, \xi}^{\tau}(t), \gamma_{x, \xi}^{\tau}(s)) = \int_0^{\infty} e^{i \theta 
\psi((\gamma_{x, \xi}^{\tau}(t), \gamma_{x, \xi}^{\tau}(s))} A((\gamma_{x, \xi}^{\tau}(t), \gamma_{x, \xi}^{\tau}(s)), \theta) d \theta, $$
obtained by pulling back the parametrix \eqref{PITAU}.

Furthermore, Theorem 9.1 of \cite{BoGu} assets that,  for any Toeplitz operator $\Pi_{\tau} Q \Pi_{\tau}$ of order $r$, 
  $$(\gamma_{x, \xi}^{\tau})^* \Pi_{\tau} Q \Pi_{\tau} \;\; \mbox{is a Toeplitz operator of order r on } \; \partial S_{\tau}, $$
whose principal symbol is the restriction of $\sigma_Q$ to $\R_+ \gamma_{x, \xi}^{\tau}$.

\subsection{Proof of Lemma \ref{EQUIV}}

\begin{proof}  
The right side is the same
as
\begin{equation} \label{EQUIV2} \begin{array}{l}
\Pi_{\tau} \overline{\sigma_{\tau t} }g^{-t}  \Pi_{\tau}  g^t g^{-t}
(\gamma_{x, \xi}^{\tau})^{**}  Op_{\lambda_j}(a) (\gamma_{x, \xi}^{\tau})^* g^t g^{-t}
\Pi_{\tau}  g^t g^{-t}\sigma_{\tau t} g^{t}  \Pi_{\tau}  \end{array} \end{equation} We use the obvious intertwining
relations
\begin{equation} \label{OBVIOUS} T^t  (\gamma_{x, \xi}^{\tau})^* =   (\gamma_{x, \xi}^{\tau})^* g_{\tau}^t, 
\;\;\;\;  (\gamma_{x, \xi}^{\tau})^{**}  T^{-t} = g^{-t} (\gamma_{x, \xi}^{\tau})^{**}    \end{equation}
 to get
\begin{equation} \label{SIMP}  g^{-t}
(\gamma_{x, \xi}^{\tau})^{**}  Op_{\lambda_j}(a) (\gamma_{x, \xi}^{\tau})^* g^t 
= 
(\gamma_{x, \xi}^{\tau})^{**}  T_t^* Op_{\lambda_j}(a) T^t (\gamma_{x, \xi}^{\tau})^*,  \end{equation}
and thus it suffices to prove 
\begin{equation} \label{EQUIV3}
\Pi_{\tau} ((\gamma_{x, \xi}^{\tau})^*)^* T^{t*} Op_{\lambda}(a) T^t (\gamma_{x, \xi}^{\tau})^* \Pi_{\tau} \simeq 
\Pi_{\tau} \overline{\sigma_{\tau t} }g^{-t}  \Pi_{\tau}  g^t
(\gamma_{x, \xi}^{\tau})^{**}  T_t^* Op_{\lambda_j}(a) T^t (\gamma_{x, \xi}^{\tau})^*
 g^{-t}
\Pi_{\tau}  g^t g^{-t}\sigma_{\tau t} g^{t}  \Pi_{\tau} . \end{equation}
Now $  T_t^* Op_{\lambda_j}(a) T^t$ appears on both sides of the purported formula and is a completely
general pseudo-differential operator. Hence \eqref{EQUIV3} is equivalent to
\begin{equation} \label{EQUIV4}
\Pi_{\tau} ((\gamma_{x, \xi}^{\tau})^* )^* Op_{\lambda}(b)  (\gamma_{x, \xi}^{\tau})^* \Pi_{\tau}  \simeq 
\Pi_{\tau} \overline{\sigma_{\tau t} }g^{-t}  \Pi_{\tau}  g^t
(\gamma_{x, \xi}^{\tau})^{**}  Op_{\lambda}(b) (\gamma_{x, \xi}^{\tau})^*
 g^{-t}
\Pi_{\tau}  g^t g^{-t}\sigma_{\tau t} g^{t}  \Pi_{\tau} , \end{equation}
for any $Op_{\lambda}(b)$.

We further observe that $g^{-t} \sigma_{\tau t} g^t = \sigma_{\tau t} \circ g^{-t}$. 
We thus need to show that
\begin{equation} \label{EQUIV5}
\Pi_{\tau} ((\gamma_{x, \xi}^{\tau})^*)^* Op_{\lambda}(b)  (\gamma_{x, \xi}^{\tau})^* \Pi_{\tau}  \simeq 
\Pi_{\tau} \overline{\sigma_{\tau t} } \Pi_{\tau} ^t
(\gamma_{x, \xi}^{\tau})^{**}  Op_{\lambda}(b) (\gamma_{x, \xi}^{\tau})^*
\Pi_{\tau}^t (\sigma_{\tau t} \circ g^{-t} ) \Pi_{\tau} . \end{equation}

To prove this, we show that

\begin{equation} \label{EQUIV6} \begin{array}{l} 
\Pi_{\tau} \overline{\sigma_{\tau t} } \Pi_{\tau} ^t
(\gamma_{x, \xi}^{\tau})^{**}  Op_{\lambda}(b) (\gamma_{x, \xi}^{\tau})^*
\Pi_{\tau}^t (\sigma_{\tau t} \circ g^{-t} ) \Pi_{\tau} \\ \\  \simeq
\Pi_{\tau} \overline{\sigma_{\tau t} } \Pi_{\tau} ^t 
\Pi_{\tau} 
(\gamma_{x, \xi}^{\tau})^{**}  Op_{\lambda}(b) (\gamma_{x, \xi}^{\tau})^*
\Pi_{\tau}  \Pi_{\tau}^t (\sigma_{\tau t} \circ g^{-t} ) \Pi_{\tau}  .
\end{array} \end{equation}
Here, we inserted an extra factor of $\Pi_{\tau}$ to the right of the first $\Pi_{\tau}^t$ and to the left of the second. 
To prove \eqref{EQUIV6} we use the $\Box_{b}^{\tau} = \dbar_b^{\tau *} \dbar_b^{\tau}$ operator on $\partial M_{\tau}$,
where $\dbar^{\tau}_b$ is the Cauchy-Riemann operator associated to the CR structure of $\partial M_{\tau}$
as a real hypersurface in $M_{\tau}$. Thus, $\Pi_{\tau}$ is the orthogonal projection onto the kernel of $\Box_b^{\tau}$. 
Since $\partial M_{\tau}$ is strictly pseudo-convex, there is a spectral gap between the $0$ eigenvalue of 
$\Box_b^{\tau}$ and its first positive eigenvalue.  Thus there exists a pseudo-differential Green's operator
$G_{\tau}$ of order $-2$ such that
$$ \Box_b^{\tau} G_{\tau} = Id - \Pi_{\tau}. $$
We use this to write 
$$\Pi_{\tau} \sigma \Pi_{\tau}^t = \Pi_{\tau} \sigma \Pi_{\tau}^t \Pi_{\tau} + 
 \Pi_{\tau} \sigma \Pi_{\tau}^t \Box_b^{\tau} G_{\tau}. $$
We then observe that the principal symbol of the zeroth order Toeplitz operator
$\Pi_{\tau} \sigma \Pi_{\tau}^t \Box_b^{\tau} G_{\tau}$ vanishes. Indeed, since $\Pi_{\tau} \Box_b^{\tau} = 0$
it is the same as $\Pi_{\tau} [\sigma \Pi_{\tau}^t, \Box_b^{\tau}] G_{\tau}$. But the commutator
$ [\sigma \Pi_{\tau}^t, \Box_b^{\tau}] $ is a Toeplitz operator adapted to the identity map on $\Sigma_{\tau}$
with vanishing principal symbol, since scalar functions commute with the symbol of $\Pi_{\tau}^t$.
It follows that 
$\Pi_{\tau} \sigma \Pi_{\tau}^t \Box_b^{\tau} G_{\tau}$ is of order $-1$, and does not contribute to the principal
symbol of the right side of \eqref{EQUIV6}, proving the claim. 

Thus we reduce to proving that the symbol of  right side of \eqref{EQUIV6} is the same as the symbol of the left side
of \eqref{EQUIV5} .  But this follows from the fact that $$\Pi_{\tau} \overline{\sigma_{\tau t} } \Pi_{\tau} ^t \Pi_{\tau}
\simeq \Pi_{\tau}, \;\;\; \Pi_{\tau} \sigma_{\tau t}  \circ g^{-t}  \Pi_{\tau} ^t \Pi_{\tau} \simeq \Pi_{\tau}. $$
Indeed, the symbol of the left side of the first expression equals $$\overline{\sigma_{\tau t}}  \sigma_{\Pi_{\tau} } \sigma_{\Pi_{\tau}^t}
\sigma_{ \Pi_{\tau}} = \overline{\sigma_{\tau t}}  \langle e_{\Lambda_{\tau}} , e_{\Lambda_{\tau}}^t \rangle 
\sigma_{\Pi_{\tau}} = \sigma_{\Pi_{\tau}} $$ since  the numerical factor cancels by \eqref{sigmataut}.
The symbol on the left side of the second is
 $$\sigma_{\tau t}  \circ g^{-t}  \langle e_{\Lambda_{\tau}} , e_{\Lambda_{\tau}}^t \rangle 
\sigma_{\Pi_{\tau}} = \sigma_{\Pi_{\tau}} .$$
Indeed, 
$g^{-t} \Lambda_t = \Lambda$ so
$$\sigma_{\tau t}  \circ g^{-t}   =   \langle e_{\Lambda}, e_{\Lambda_t} \rangle^{-1}
= \overline{\sigma_{\tau t}}, $$
and so the claim follows as for the first expression. This completes the proof of Lemma \ref{EQUIV}
and therefore of Proposition \ref{INTERTWINING2}.

\end{proof}

Differentiating Lemma \ref{EQUIV} in $t$ and setting $t = 0$ gives the infinitesimal version: 
\begin{cor} \label{INF}  Let $ \sigma_0 = \frac{d}{dt}|_{t = 0}  \langle   e_{\Lambda_t},
e_{\Lambda} \rangle^{-1}.  $  There exists a  pseudo-differential operator $ R_{-1}$  of order
 $-1$  so that,

$$\Pi_{\tau} ((\gamma_{x, \xi}^{\tau})^*)^*   Op_{\lambda}([\frac{\partial}{\partial t}, a])  (\gamma_{x, \xi}^{\tau})^* \Pi_{\tau}  \simeq
\Pi_{\tau}  [\Xi_{\sqrt{\rho}}  + \sigma_0 + R_{-1}),
(\gamma_{x, \xi}^{\tau})^{**}  Op_{\lambda_j}(a) (\gamma_{x, \xi}^{\tau})^*]
\Pi_{\tau} .$$

\end{cor}

\begin{proof}
$$\frac{d}{dt}_{t = 0} \Pi_{\tau} g^t \sigma_{t, \tau} \Pi_{\tau} = \Pi_{\tau} (\Xi_{\sqrt{\rho}} + \frac{d}{dt}|_{t = 0}
\sigma_{t, \tau} ) \Pi_{\tau}. $$ By \eqref{sigmataut}, the second term is
$$ \frac{d}{dt}|_{t = 0}  \sigma_{\tau, t} =  \frac{d}{dt}|_{t = 0}    \langle e_{\Lambda_t}, e_{\Lambda} \rangle^{-1} .
$$

\end{proof}


 Alternatively, one can  construct a paramatrix for $ (\gamma_{x, \xi}^{\tau})^*  \Pi_{\tau} $ by pulling
back the parametrix \eqref{PITAU}, to get 
$$\Pi_{\tau}(\gamma_{x, \xi}(t),\zeta) = \int_0^{\infty} e^{i \theta \psi((\gamma_{x, \xi}^{\tau}(t), \zeta)} A((\gamma_{x, \xi}^{\tau}(t), \zeta, \theta) d \theta. $$ Applying $\frac{\partial}{\partial t}$ changes the amplitude to 
$$i\theta \frac{\partial}{\partial t} \psi((\gamma_{x, \xi}^{\tau}(t), \zeta)  A + \frac{\partial}{\partial t} A((\gamma_{x, \xi}(t), \zeta, \theta) . $$
The first term is $\partial \psi ((\gamma_{x, \xi}^{\tau}(t), \zeta)  \cdot \gamma_{x, \xi}^{\tau'}(t)$, and on
the diagonal $\partial \psi$ is the contact form, which evaluates to $1$ on  $\gamma_{x, \xi}^{\tau'}(t)$. Similarly,
the symbol of $\Xi_{\sqrt{\rho}}$ evaluates to $\theta$.

\section{  \label{LLPROOF} Lebesgue limits of matrix elements: Proof of Proposition \ref{LL}}

In this section, we use the intertwining relation in  Proposition \ref{INTERTWINING2} to prove that matrix elements  
\begin{equation} \label{ME} \langle a U_j^{I_{\tau}  x, \xi}, U_j^{I_{\tau}, x, \xi} \rangle \end{equation}
of compactly supported multiplication operators on  $\partial S_{\tau}$ with respect to the $L^2$ normalized
pullbacks of  Definition \ref{Ujdef}  are asymptotically invariant
under translation $T_t: L^2(\R) \to L^2(\R)$. It is simplest to show 
\begin{prop} \label{INFINV}  Let $a \in C_0^{\infty}(I_{\tau})$. Then
$$\langle  \frac{\partial}{\partial t} a,  |U_j^{I_{\tau}, x, \xi} |^2 \rangle = o_{I_{\tau}}(1), \;\; \lambda_j \to \infty. $$
\end{prop}

The invariance property scales with the  normalization of the pullbacks.   We therefore use the preliminary normalization
\begin{equation} u_j^{\tau} = \frac{\phi_j^{\C}}{||\phi_j^{\C}||_{L^2(\partial M_{\tau})}}
 \end{equation}
and later divide again by the mass of the pullback on $I_{\tau}$.

\begin{proof}   By differentiating Proposition \ref{VT} in $t$ (and setting $t = 0$),  there exists a pseudo-differential
operator $\sigma_0$ on $\partial M_{\tau}$ of order zero so that,
\begin{equation}  \Pi_{\tau} (\Xi_{\sqrt{\rho}} + \sigma_{0} )\Pi_{\tau} u_j^{\tau} = i \lambda_j u_j^{\tau}. \end{equation}
We then use Corollary \ref{INF} to obtain,

\begin{equation} \begin{array}{lll}   \langle (\frac{\partial}{\partial t} a )   (\gamma_{x, \xi}^{\tau})^* u_j^{\tau} ,    (\gamma_{x, \xi}^{\tau})^* u_j^{\tau} \rangle_{\R}  & = &\langle [\frac{\partial}{\partial t} , a]    (\gamma_{x, \xi}^{\tau})^* u_j^{\tau} ,    (\gamma_{x, \xi}^{\tau})^* u_j^{\tau} \rangle_{\R} 
\\ &&\\ & = & \langle
 [\Xi_{\sqrt{\rho}}  + \sigma_0 + R_{-1}),
(\gamma_{x, \xi}^{\tau})^{**}  Op_{\lambda_j}(a) (\gamma_{x, \xi}^{\tau})^*] u_j^{\tau} , u_j^{\tau}  \rangle_M \\ && \\ 
& = &\langle
 [ R_{-1},
(\gamma_{x, \xi}^{\tau})^{**}  Op_{\lambda_j}(a) (\gamma_{x, \xi}^{\tau})^*] u_j^{\tau} , u_j^{\tau}  \rangle_M
\end{array} \end{equation}
where $R_{-1}$ 
is a  pseudo-differential operator of order $-1$.  We then apply the Cauchy-Schwartz
inequaltity and the fact that
$$||R_{-1} u_j^{\tau} ||_{L^2(\partial M_{\tau}} \leq \lambda_j^{-1} ||u_j^{\tau}||_{L^2(\partial M_{\tau}}$$ to bound the final expression by 
$$\begin{array}{lll}   |
|\langle
 [ R_{-1},
(\gamma_{x, \xi}^{\tau})^{**}  Op_{\lambda_j}(a) (\gamma_{x, \xi}^{\tau})^*] u_j^{\tau} , u_j^{\tau}  \rangle_M| &  \leq & \lambda_j^{-1}
||  \Pi_{\tau} (\gamma_{x, \xi}^{\tau})^*   a_0      (\gamma_{x, \xi}^{\tau})^* \Pi_{\tau}  u_j^{\tau} ||_{L^2(\partial M_{\tau})} \\&&\\&  \leq & \lambda_j^{-1}
||      (\gamma_{x, \xi}^{\tau})^* u_j^{\tau} ||_{L^2(I)},\;\;\; (I = \mbox{supp} a). 
\end{array}$$
If we divide $\gamma_{x,\xi}^{\tau * }u_j^{\tau}$ by $||      (\gamma_{x, \xi}^{\tau})^* u_j^{\tau} ||_{L^2(I)}$
to obtain $U^{I, x, \xi}_j$ then the latter expression tends to zero as $\lambda_j \to \infty$ and 
the sequence $|U^{I, x, \xi}_j|^2$ is a sequence of probability measurs on $I$ whose weak* limits
must be probability measures on $I$ given by constant multiples of $dt$. Of course, there is only
one such probability measure.


\end{proof}

\begin{rem}

If we chose to divide by the $L^2$ norm of  $\gamma_{x,\xi}^{\tau * }u_j^{\tau}$  on a larger  interval $J$, $ I \subset J$,
 then we would end up with the ratio $\frac{||      (\gamma_{x, \xi}^{\tau})^* u_j^{\tau} ||_{L^2(I)}}{||      (\gamma_{x, \xi}^{\tau})^* u_j^{\tau} ||_{L^2(J)}} < 1$ and the same estimate holds.
\end{rem}

We give a second proof using the global propagator. Although it is essentially the same, it is
worth recording because the  interval on which the remainder estimate is made gets shifted by
$t$ units.  
We begin with a  pointwise Weyl law giving almost uniform bounds for restrictions of  `most'   
normalized complexified eigenfunctions.

\begin{lem} \label{DENSITY1} For all  compact intervals $I$ and $\epsilon > 0$ there exists $C_{I, \epsilon} >  0$ and a subsequence $\scal_{I, \epsilon}$ of counting density $D^*(\scal_{I, \epsilon}) \geq 1 - \epsilon$
so that
$$
 \limsup_{j \to \infty, j \in \scal_{I, \epsilon}}  e^{- 2 \tau \lambda_j}
\int_I |\phi_{\lambda_j}^{\C}(\gamma_{x, \xi}(t + i \tau))|^2 dt \leq  C_{I, \epsilon}  |I|. $$
\end{lem}

Here, as above, $|I|$ is the Lebesgue measure of $I$.

\begin{proof} 

It follows from Proposition \ref{PWa} that
for all $t \in \R$, 
$$\frac{1}{M(\lambda)} \sum_{j: \lambda_j \in  I_{\lambda}} e^{- 2 \tau \lambda_j}
|\phi_{\lambda_j}^{\C}(\gamma_{x, \xi}(t + i \tau))|^2   = O(\lambda^{-1}), $$
and by integrating over a
compact interval $I \subset \R$ we have
\begin{equation} \label{LWL} \frac{1}{M(\lambda)} \sum_{j: \lambda_j \in  I_{\lambda}} e^{- 2 \tau \lambda_j}
\int_I |\phi_{\lambda_j}^{\C}(\gamma_{x, \xi}(t + i \tau))|^2 dt = |I| + O(\lambda^{-1}), \end{equation}
where the remainders are uniform in $t$ resp. $I$.
We then apply a simple  Chebyshev inequality to \eqref{LWL}. 
For any sequence  $\{X(j)\}$  of positive real numbers satisfying
$$M \{X(j)\}: = \lim_{N \to \infty} \frac{1}{N} \sum_{j = 1}^N X(j)= M, $$
we have
$$\begin{array}{l} D^*\{j:  X(j)  \geq
  M+  R\}  \leq \frac{1}{M + R} M \{ X(j)\}  =
 \frac{M}{M + R}, \\ \\ \;\; \mbox{hence}\;\;
D^*\{j:  X(j) \leq  M+  R\}  \geq
 \frac{R}{M + R}.  \end{array}$$ 
We put
$$X(j) = e^{- 2 \tau \lambda_j} |\phi_{\lambda_j}^{\C}(\gamma_{x, \xi}(t + i \tau))|^2 , \;R = \frac{1}{\epsilon}, \;\;C_{I, \epsilon} = M + R.$$

\end{proof}





Lemma \ref{DENSITY1} implies Proposition \ref{INFINV} as follows:

\begin{proof}

We assume that  $a \in C_c^{\infty}(\R)$. 
By Lemma \ref{EQUIV} ,  we have 
\begin{equation} \begin{array}{lll}   \langle a    (\gamma_{x, \xi}^{\tau})^* u_j^{\tau} ,    (\gamma_{x, \xi}^{\tau})^* u_j^{\tau} \rangle_{\R}  & = &  \langle a (\gamma_{x, \xi}^{\tau})^* V_{\tau}^t u_j^{\tau}, 
(\gamma_{x, \xi}^{\tau})^{*} V_{\tau}^t  u_j^{\tau} \rangle_{\R} \ \\ && \\ & = & 
 \langle V_{\tau}^{*t} (\gamma_{x, \xi}^{\tau})^{**}  a (\gamma_{x, \xi}^{\tau})^*  
 V_{\tau}^t u_j^{\tau} .
  u_j^{\tau} \rangle_{\partial M_{\tau}} \\ &&\\& = &
  \langle ( (\gamma_{x, \xi}^{\tau})^* )^* T_t^* a  T_t (\gamma_{x, \xi}^{\tau})^* u_j^{\tau}, u_j^{\tau} \rangle_{\partial M_{\tau}} + O(|| T_t(a) (\gamma_{x, \xi}^{\tau})^*  R_{-1}( \lambda_j)  u^{\tau}_j||^2_{\partial M_{\tau}} )\;\;\; \\ &&\\&=&
  \langle ( (\gamma_{x, \xi}^{\tau})^* )^*(T_t a) (\gamma_{x, \xi}^{\tau})^* u_j^{\tau}, u_j^{\tau} \rangle_{\partial M_{\tau}} +  O(|| (T_t a)  (\gamma_{x, \xi}^{\tau})^*  R_{-1}(\lambda_j) u^{\tau}_j||^2_{\partial M_{\tau}} )\;\;\\ &&\\&=&
  \langle (  (T_t a) (\gamma_{x, \xi}^{\tau})^* u_j^{\tau},  (\gamma_{x, \xi}^{\tau})^*  u_j^{\tau} \rangle_{\R} \\
&&\\  &+&  \lambda_j^{-1} O(|| (\gamma_{x, \xi}^{\tau})^*   u^{\tau}_j||_{T_t(I_{\tau})} ),\;\;\; (\mbox{as}\; \lambda_j \to \infty) 
\end{array} \end{equation}
where (as above)  $R_{- 1}$ is a pseudo-differential operator of order $- 1$ and 
 $T_t a(s) = a(s + t). $ 

We then divide by $(|| (\gamma_{x, \xi}^{\tau})^*   u^{\tau}_j||_{I_{\tau}}$, and observe that
we get  translation  invariance on the longest interval with the property  that
$$\frac{|| (\gamma_{x, \xi}^{\tau})^*   u^{\tau}_j||_{T_t(I_{\tau})} }{|| (\gamma_{x, \xi}^{\tau})^*   u^{\tau}_j||_{I_{\tau}}}$$
is bounded,  or more generally is $o(\lambda_j)$. Lemma \ref{DENSITY1} implies that this happens for a subsequence of
eigenvalues of full density.
This concludes the proof of Proposition \ref{LL}.

\end{proof}

\subsection{\label{LLYSECT} Lebesgue limits for the family of translates}

When dealing with non-periodic geodesics it is useful to consider the family
$$\gamma_{G^s(x, \xi)}^* \phi_j, \;\; (s \in \R) $$
of pullbacks as $s, \lambda_j$ vary. We have,
$$\gamma_{G^s(x, \xi)}^* \phi_j (t)  = \gamma_{x, \xi}^* \phi_j(t - s), $$
so we are considering the family of translates 
$$\fcal = \{ T_s \gamma_{x, \xi}^* \phi_j, \;\;\; j = 1, 2, \dots, s \in \R\}. $$
The family of translates for fixed $j$ is of course not compact in $C_b(\R)$ for general $\gamma_{x, \xi}$.  On the other
hand, for fixed $j$ the family of translates $\{\phi_j(G^s(x, \xi)), \;\; s \in \R\}$ is compact in $C(S^* M)$. 

We  modify Definition \ref{Ujdef} as follows: 
\begin{defin}  \label{Yjdef}   Let $\{N_j\} \subset \R$ and define
\begin{equation}  Y_j^{\tau, T, x, \xi}  := U_j^{\tau, T, G^{N_j}(x, \xi)} \end{equation}
\end{defin}
We then modify the proof of Proposition \ref{LL}  to show that the weak* limits of 
$$ \left|  Y_j^{\tau, T, x, \xi}
\right|^2 dt\;\; \mbox{
on the line segments}  \;\;\; \partial S_{\tau, T} $$
tend to normalized Lebesgue measure:

\begin{prop} \label{LLY} (Lebesgue limits for moving pullbacks)  Let $(x, \xi) \in B_{\epsilon}^* M$ be any point. Then  as long 
as $\gamma_{x, \xi}^* \phi_j \not=  0$ (identically),  the 
 sequence $\{| Y_j^{\tau, T, x, \xi}|^2\}$  is QUE on $\R$ with limit measure given by normalized  Lebesgue measure on each 
segment
$\partial S_{\tau, T} $ 

\end{prop}
 
\begin{proof}  As  in the proof of Proposition \ref{LL}, the key step is to generalize the intertwining relation in  Proposition \ref{INTERTWINING2} to prove that matrix elements  
\begin{equation} \label{MEY} \langle Op_{\lambda_j}(a) Y_j^{\tau, T, x, \xi}, Y_j^{\tau, T, x, \xi} \rangle \end{equation}
  are asymptotically invariant
under translation $T_t: L^2(\R) \to L^2(\R)$.  We need to  replace $\gamma_{x, \xi}^{\tau *}$
by $\gamma_{G^{N_j}(x, \xi)}^{\tau *}$ everywhere in the proof.  In fact, the intertwining relation generalizes
to the two-parameter equivalence
$$\Pi_{\tau} ((\gamma_{G^{s}(x, \xi)}^{\tau})^*)^*  T^{t*} Op_{\lambda}(a) T^t (\gamma_{G^{s}((x, \xi)}^{\tau})^* \Pi_{\tau}  \simeq
\Pi_{\tau} \overline{\sigma_{\tau t} }g^{-t}  \Pi_{\tau} 
(\gamma_{G^{s}((x, \xi)}^{\tau})^{**}  Op_{\lambda_j}(a) (\gamma_{G^{s}(x, \xi)}^{\tau})^*
\Pi_{\tau} \sigma_{\tau t} g^{t}  \Pi_{\tau} ,$$
where the remainders are uniform in $s$. To see that the remainders are uniform we observe that the intertwining
relation holds for all $(x, \xi) \in S^*M $ with uniform remainders in $(x, \xi)$ since $S^* M$ is compact, and
that we are only specializing the estimate to  points on $\gamma_{x, \xi}$.

We then set $s = N_j$. Due to the uniformity of the remainders in the intertwining relation, the matrix elements relative 
relative to $\gamma_{G^{N_j}(x, \xi)}^* u_j^{\tau}$ are of one lower order than the main terms, and hence are of
$O(\lambda_j^{-1})$ when taken against the $L^2$-normalized $Y_j^{\tau, T, x, \xi}. $

\end{proof}

\section{\label{HARTOGSsect} Growth rates of restricted eigenfunctions: Proof of Lemma \ref{HARTOGSIntro}}

We apply a general compactness theorem for subharmonic functions (see
\cite[Theorem~4.1.9]{Ho}):
\medskip

\bigskip

{\it Let $\{v_j\} \subset SH(X) $ be a sequence of subharmonic functions in an
open set $X \subset \R^m$ which have a uniform upper bound on any
compact set. Then either $v_j \to -\infty$ uniformly on every
compact set, or else there exists a subsequence $v_{j_k}$ which 
converges in $L^1_{loc}(X)$ to some   $v \in L^1_{loc}(X) \cap SH(X)$. Further,  we have (`Hartogs' Lemma):
\begin{equation} \label{SH}  \left\{ \begin{array}{ll} (i)  &  \limsup_j
v_j(x) \leq v(x) \;\; \mbox{ with equality almost everywhere} \\ & \\
(ii) &  \mbox{For every
compact subset}\; K \subset X \; \mbox{ and every continuous function f},\\ & \\ &
\limsup_{j \to \infty} \sup_K (v_j - f) \leq \sup_K (v - f). \\ & \\
(iii) & \mbox{In particular, if}\; f \geq v\;; \mbox{ and} \; \epsilon > 0, \; \mbox{then } \; v_j \leq f
+ \epsilon \; \mbox{ on} \; K \; \mbox{ for j large enough.} \end{array} \right. \end{equation}}

In \cite{Z} we applied this compactness theorem to prove that for the full sequence of ergodic eigenfunctions,
\begin{equation} \label{torho} v_j = \frac{1}{\lambda_j} \log |\phi_j^{\C}(\zeta)|^2 \to 2  \sqrt{\rho} \;\; \mbox{in}\;\;
L^1(M_{\epsilon}). \end{equation} . 
We note that when $v_{j_k} \to v$ in $L^1_{loc}$ then $\limsup_{k \to \infty} v_{j_k}$ need not
be upper semi-continuous. If we denote by $v^*$ its upper semi-continuous regularization then
$v = v^*$ almost everywhere and by the compactness theorem $\limsup_k v_{j_k}  = v$ almost everywhere.



\begin{proof}

 To prove Lemma \ref{HARTOGSIntro}, we   first observe that 
$v_j$  is
SH (subharmonic) on $S_{\epsilon}$, and  apply the abo e theorem to $\{v_j\}$  with $X = S_{\epsilon_0}$. 
Exactly as in
\cite{Z}),  it follows from Proposition \ref{PWa} that the sequence $\{v_j\}$  is uniformly
bounded above by $ 2 |\tau|$  globally on $M_{\tau}$.  Under the condition of non-uniform convergence to $-\infty$, 
 there must exist a subsequence, which we continue to denote by
$\{v_{j_k}\}$, which converges in $L^1_{loc}(S_{\epsilon_0})$ to
some $v \in L^1_{loc}(S_{\epsilon_0}) \cap SH(S_{\epsilon_0}).$

Thus,  $v \leq 2 |\tau|$. 
Assume for purposes of contradiction  that $v < 2 |\tau| - \epsilon$ on an open set $W_{\epsilon} \subset S_{\epsilon}$.
Let $I_{\tau}^{\epsilon} =
W_{\epsilon} \cap \partial S_{\tau}$.
By Lemma \ref{LL}, one has $||U^{x, \xi, I_{\tau}}_{j_k}|^2 \to \frac{1}{I_{\tau}} d t$ for every interval $I_{\tau}$. In particular,
this holds on $I_{\tau}^{\epsilon}$. 
We claim that $v \leq  2 |\tau| - \epsilon$ on all of $\partial S_{\tau}$. If not, there is an interval $I_{\tau, \delta}$ where
$v \geq 2 |\tau| - \epsilon + \delta$ for some $\delta > 0$.
 That is,
$$\gamma_{x, \xi}^{* \tau} u_j^{\tau} \leq e^{(2 |\tau| - \epsilon) \lambda_j} \;\; \mbox{on}\;\; I_{\tau}^{\epsilon}, \;\; \mbox{and}\;\;
v  \geq  2 |\tau| - \epsilon + \delta  \;\; \mbox{on}\;\; I_{\tau, \delta}. $$
For any interval $I$, 
$$\frac{1}{\lambda_{j_k}} \log \int_I |u_{j_k}^{\tau}|^2 \frac{dt}{|I|} \geq 2 \int_I v_{j_k} \frac{dt}{|I|} \to 
2 \int_I v \frac{dt}{|I|}. $$
So if $I = I_{\tau, \delta}$ we have 
$$\frac{1}{\lambda_{j_k}} \log \int_{ I_{\tau, \delta}} |u_{j_k}^{\tau}|^2 \frac{dt}{| I_{\tau, \delta}|}  \geq 2  (|\tau| - \epsilon + \delta/2)$$
for $k \geq k_0(\delta)$. That is,
$$ \int_{ I_{\tau, \delta}}|u_{j_k}^{\tau}|^2 \frac{dt}{| I_{\tau, \delta}|} \geq e^{2\lambda_{j_k} (  |\tau| - \epsilon + \delta/2)}. $$

 Again we have
 $|U^{x, \xi, I_{\tau, \delta}}_{j_k}|^2 \to \frac{1}{I_{\tau, \delta}} d t$. Now choose $T$ large enough so that both $I^{\epsilon}_{\tau}$
and $I_{\tau, \delta}$ are contined in $I_{T} = [-T,T]$. Then we also have
$$ |U^{x, \xi, I_{T}}_{j_k}|^2 dt \to dt, \;\; \mbox{weak * on}\;  [-T,T]. $$
Since $I_{\tau, \delta} \subset I_T$, for sufficently large $k$,
$$||\gamma_{x, \xi}^* u_{j_k}^{\tau}||_{L^2(I_T)} \geq \frac{|I_{\tau, \delta}|}{2 T}   ||\gamma_{x, \xi}^* u_{j_k}^{\tau}||_{
L^2(I_{\tau, \delta})} \geq 
\frac{|I_{\tau, \delta} | }{2 T}\; e^{(\tau - \epsilon + \delta/2) \lambda_{j_k}}. $$
We then compare the two statements,
$$ |U^{x, \xi, I_{\tau}^{\epsilon}}_{j_k}|^2 dt \to \frac{dt}{|I|} \;\; \mbox{weak * on}\;\; I_{\tau}^{\epsilon}, \;\;  |U^{x, \xi, I_{T}}_{j_k}|^2 dt
\to \frac{dt}{|I_T|}\;\; \mbox{weak * on}\;\; I_{T}. $$
The conditions overlap for $a \in C_c^{\infty}(I_{\tau}^{\epsilon})$, so let us choose a test function approximating
the characteristic function ${\bf 1}_{I_{\tau}^{\epsilon}}$, and in fact we may assume $a$ equals this test function. 
But on $I_{\tau}^{\epsilon}$,
$$ |U^{x, \xi, I_T}_{j_k}|^2 = |U^{x, \xi, I_{\tau}^{\epsilon}}_{j_k}|^2 
 \frac{||\gamma_{x, \xi}^{\tau *} u_j^{\tau}||_{L^2(I_{\tau}^{\epsilon})}}{||\gamma_{x, \xi}^{\tau *} u_j^{\tau}||_{L^2(I_T)}}  \frac{|I_T|}{|I_{\tau}^{\epsilon}|}   \leq C_{T, \epsilon, \delta} e^{- \half \delta \lambda_j}  |U^{x, \xi, I_{\tau}^{\epsilon}}_{j_k}|^2.  $$
It is impossible that both $  |U^{x, \xi, I_T}_{j_k}|^2   $  and $ |U^{x, \xi, I_{\tau}^{\epsilon}}_{j_k}|^2 $
tend to 1 weakly on $I_{\tau}^{\epsilon}$.

This contradiction shows that \begin{equation}
\label{V} v <  2 |\tau| - \epsilon  \; \mbox{on}   \; W_{\epsilon}  
\implies v < 2 |\tau| - \epsilon \;\; \mbox{on}\;\; \partial S_{\tau}. \end{equation}
It follows that
\begin{equation} \label{SMALL}
 \limsup_{k \to \infty} \; \frac{1}{ \lambda_{j_k}}
\ \log \left| \gamma_{x, \xi}^* \phi_{\lambda_{j_k}}^{\C} (t + i \tau)
\right|^2 \leq 2(|\tau| - \epsilon) \;\; \mbox{on all of} \;\partial
S_{\tau}. \end{equation} 






\end{proof}

\subsection{Hartogs theorem for the family of translates}

In this section, we prove a slightly more general version in which the origin $(x, \xi)$ is allowed to
move with the index of the sequence:

\begin{lem} \label{HARTOGSM} For any compact analytic Riemannian manifold $(M, g)$ of any dimension $m$,  any
complexified geodesic $\gamma_{x, \xi}$, and sequence $\{N_j\} \subset \R$ and and any sequence of eigenfunctions, 
the family of plurisubharmonic functions $$v_{j} : = \frac{1}{\pi \lambda_{j}} \ \log \left| \gamma_{G^{N_j}(x, \xi)}^* \phi_{\lambda_{j}}^{\C} (t +
i \tau) \right|^2 $$ is precompact in 
$L^1_{loc}(S_{\epsilon})$ as long as 
it does not converge  uniformly to $-%
\infty$ on all compact subsets of $S_{\tau}$. Moroever, any limit of a subsequence   is $\leq  2 |\tau|$. If the
 limit $v$  of a subsequence  $v_{j_k}$   is $\leq 2 \tau - \epsilon$ on an open interval $t \in (a, b)$, then  $v  \leq 2 |\tau| - \epsilon$
for all $t \in \R$ and 
$$\limsup_{k \to \infty} v_{j_k} \leq 2  |\tau| - \epsilon. $$
\end{lem}

\begin{proof}   

The uniform upper bound by $2 |\tau|$ of course holds for the whole family $\fcal$ \eqref{FAMILYs}. 
 We also use Lemma \ref{LLY} in place of Proposition \ref{LL} and with $Y_{j_k}^{x, \xi, I}$ in 
place of $U_{j_k}^{x, \xi, I}$ where $I$ is any one of the intervals in the proof of Lemma \ref{HARTOGSIntro}, and
then follow the same argument.

\end{proof}

\section{Periodic geodesics on surfaces: Proof of Proposition \ref{MAINPROP}}

In this section we  prove Proposition \ref{MAINPROP} and therefore Theorem \ref{MAINCOR}  for restrictions
to periodic geodesics on real analytic surfaces.
We thus assume that $\gamma_{x, \xi}$ is a periodic orbit of period $L$.  We then denote the orbital Fourier
coefficients  of an eigenfunction  by
$$\nu_j^{x, \xi}(n) = \frac{1}{L_{\gamma}} \int_0^{L_{\gamma}} \phi_j(\gamma_{x, \xi}(t)) e^{- \frac{2 \pi i n t}{L}} dt. $$
Thus, we have
\begin{equation} \label{PERb} \phi_j(\gamma_{x, \xi}(t) )= \sum_{n \in \Z}  \nu_j^{x, \xi}(n)  e^{\frac{2 \pi i n t}{L}}. 
\end{equation}
Hence the analytic continuation is given by 
\begin{equation} \label{ACPERb} \phi^{\C}_j(\gamma_{x, \xi}(t + i \tau) )= \sum_{n \in \Z}  \nu_j^{x, \xi}(n)  e^{\frac{2 \pi i n (t + i \tau)}{L}}. \end{equation}

It follows from  the Paley-Wiener theorem and  from the fact that $\gamma_{x, \xi}^* \phi_j$ admits an
analytic continuation to the annulus $e^{- |\tau|} < r < e^{|\tau|}$ that    $|\nu_j^{x, \xi}(n) |\leq C_j e^{- |n| \tau}. $
Furthermore, the Fourier modes $|n| >> \lambda_j$ are exponentially decaying. In semi-classical language, such
high angular momentum is inconsistent with the energy $\lambda_j^2$ of the particle. 
More precisely, for  $n \geq \lambda_j$, 
\begin{equation} \label{ACSOB} |\nu_j^{x, \xi}(n)|^2  \leq \lambda_j^{\frac{m-1}{2}} e^{2 |\tau|  (\lambda_j - n)}.  \end{equation}
Indeed,  Propositon \ref{PWa} gives
$$\frac{1}{L_{\gamma}} \int_{\gamma_{x, \xi}^{\tau}} |\phi_j^{\C}|^2 ds \leq  \lambda_j^{\frac{m-1}{4}}
e^{2 |\tau| \lambda_j}. $$It follows that
$$\sum_n |\nu_j^{x, \xi}(n)|^2 e^{2 n |\tau|} \leq \lambda_j^{\frac{m-1}{4}}  e^{2 |\tau| \lambda_j}, $$
and \eqref{ACSOB} follows immediately from
$$\sum_{n \geq \lambda_j}  |\nu_j^{x, \xi}(n)|^2 e^{2 n |\tau|} \leq  \lambda_j^{\frac{m-1}{4}}  e^{2 |\tau| \lambda_j}. $$

\subsection{Mass in the highest allowed modes}

The purpose of this section is to prove
\begin{prop} \label{FCSATa}  Let $\dim M = 2$.  Suppose that $\{\phi_{\lambda_j}\}$ is QER along the periodic geodesic $\gamma_{x, \xi}$.
Then for all $\epsilon > 0$, there exists $C_{\epsilon} > 0$ so that
$$\sum_{n: |n| \geq (1 - \epsilon) \lambda_j}  |\nu_{\lambda_j}^{x, \xi}(n)|^2 \geq  C_{\epsilon}. $$
Consequently, 

$$\sum_{n: |n| \geq (1 - \epsilon) \lambda_j}  |\nu_{\lambda_j}^{x, \xi}(n)|^2 e^{- n \tau}  \geq  C_{\epsilon} e^{\tau(1 - 
\epsilon) \lambda_j}. $$

\end{prop}

\begin{proof} The first step is a direct application of the QER theorem of \cite{TZ}.  It implies that  within the range $|n| \leq \lambda_j$ the Fourier coefficients are asymptotically
all of the same size $\frac{1}{\lambda}$.

\begin{lem} \label{FCPER} Assume that $\gamma_{x, \xi}^* \phi_j$ is QER. Then for  any $0 < a < b \leq 1$ we have
$$\sum_{n: a \lambda_j  |n| \leq b \lambda_j}  |\nu_j^{x, \xi}(n)|^2 \to \int \chi_{[a, b]} d\mu_{x, \xi}. $$
\end{lem}

One can prove this using homogeneous or semi-classical pseudo-differential operators.
For instance, let  $\chi_{\lambda}(D)$ be a  semi-classical convolution
operator on the circle $S^1_L = \R / \frac{2 \pi}{L} \Z$  with $D = \frac{1}{i} \frac{d}{dt}$ and consider
$$\langle \chi_{\lambda}(D)  \gamma_{x, \xi}^* \phi_j, \gamma_{x, \xi}^* \phi_j \rangle_{S^1_L}
= \sum_{n \in \Z} \chi(\frac{2 \pi i n}{L \lambda}) |\nu_j^{x, \xi}(n)|^2. $$
Assuming $\{\phi_j\}$ satisfies QER with respect to $\gamma_{x, \xi}$, we have
$$\langle \chi_{\lambda}(D)  \gamma_{x, \xi}^* \phi_j, \gamma_{x, \xi}^* \phi_j \rangle_{S^1_L} \to \int_{B^* S^1_L}
\chi d \mu_{x, \xi}. $$

It follows that,  for any $\epsilon > 0$ we have
$$\sum_{n:  (1 - \epsilon) \lambda_j \leq  |n| \leq (1 + \epsilon) \lambda_j} \; |\nu_j^{x, \xi}(n)|^2 \to 2 \epsilon. $$
Consequently, 

$$\sum_{n: |n| \geq (1 - \epsilon) \lambda_j}  |\nu_{\lambda_j}^{x, \xi}(n)|^2 e^{-2 n \tau}  \geq  C_{\epsilon} e^{2\tau(1 - 
\epsilon) \lambda_j}. $$

\end{proof}

\subsection{Completion of proof of Proposition \ref{MAINPROP} for periodic geodesics}

The following Lemma is an integrated version of Proposition \ref{MAINPROP} .

\begin{lem} \label{L2NORM} Let $\dim M = 2$ and assume that $\{\phi_j\}$ satsifies QER along the
periodic geodesic $\gamma$. Let $||\gamma_{x, \xi}^{\tau*} \phi_j^{\C}||^2_{L^2(\partial S_{\tau}^L)}$ be the $L^2$-norm
of the complexified restriction of $\phi_j$ along  one period cell $\partial S_{\tau}^L$. Then,
$$\lim_{\lambda_j \to \infty} \frac{1}{\lambda_j} \log ||\gamma_{x, \xi}^{\tau*} \phi_j^{\C}||^2_{L^2(\partial S_{\tau})}
\to 2 |\tau|. $$
\end{lem}

Indeed, this  follows from Proposition \ref{FCSATa} since,  for any $\epsilon > 0$,
$$\liminf_{\lambda_j \to \infty} \log \sum_{n: |n| \geq (1 - \epsilon) \lambda_j}  |\nu_{\lambda_j}^{x, \xi}(n)|^2 e^{-2 n \tau}  \geq 2 |\tau|(1 - 
\epsilon). $$

To prove Proposition \ref{MAINPROP} we argue by contradiction.  If it is false, then there exists a time
interval $[a,b]$ and $\epsilon_0 > 0$  so that
 $$\int_a^b |(\gamma_{x, \xi}^{\tau})^* \phi_j^{\C}|^2 = O(e^{(2 |\tau| - \epsilon_0) \lambda_j}). $$
On the other hand, by Proposition \ref{FCSATa}  we know that over the whole period interval
we have, 
 $$\int_0^L |(\gamma_{x, \xi}^{\tau})^* \phi_j^{\C}|^2 \geq C (e^{(2 |\tau| - \epsilon) \lambda_j})$$
for all $\epsilon > 0$. 
Hence, we have
\begin{equation}  \int_a^b |U_j^{x, \xi, \tau, T}|^2 dt = O(e^{ - \epsilon_0 \;\lambda_j}).  \end{equation}
But by Proposition \ref{LL}, every weak* limit of $\{U_j^{x, \xi, \tau, T}\}$ is a constant multiple of Lebesgue
measure. It follows that the multiple must be zero. But this contradicts Lemma  \ref{L2NORM}.

\section{Non-periodic geodesics: Proof of Theorem \ref{MAINCORnonper}}

In the periodic case, a key step is to compute the $L^2$ norm of the analytic continuation
using the Plancherel theorem and to compare it to the $L^2$ norm in the real domain
using the Plancherel theorem.  For non-periodic geodesics,  we  introduce a decaying analytic
factor to put $\gamma_{x, \xi}^* \phi_j^{\C}$ into $L^2$ along horizontal lines.

\subsection{Analytic convergence factors}

Let $\gcal $ be a real analytic function whose analytic extension to $S_{\tau}$ lies
in $L^2(\partial S_{\tau}, dt)$ for each $\tau < \epsilon.$  In particular we have in mind $\gcal(x) = e^{- x^2/2}$, but a less rapidly decaying choice is 
$\gcal = \frac{1}{t+ i p}  $
 for large enough $|p| >> \epsilon$.

Thus for a given analytic and decaying convergence factor $\gcal$,  we consider 
\begin{equation} \label{nugcal}  \nu_{\lambda}^{x, \xi, \gcal} (\sigma) = \fcal \left(\gcal \cdot \gamma_{x, \xi} ^*\phi_{\lambda} ) \right)(\sigma)  = \int_{\R} 
\gcal \cdot  \phi_{\lambda} (\gamma_{x, \xi}(t))  e^{- i t \sigma} dt. \end{equation}
We then have the  Fourier inversion formula
$$\gcal \cdot  \gamma_{x, \xi}^* \phi_{\lambda_j}(s) = \int_{\R} e^{i s \sigma} \nu_{\lambda}^{x, \xi, \gcal} (\sigma) ) d \sigma,$$
and  analytically continue  $\gcal \cdot \gamma_{x, \xi}^* \phi_{\lambda_j}$ to
\begin{equation} \label{FIFg} \gcal  \cdot \gamma_{x, \xi}^* \phi_{\lambda_j}^{\C} (s + i \tau)  =
 \int_{\R} e^{i (s + i \tau) \sigma}\nu_{\lambda}^{x, \xi, \gcal} (\sigma) d \sigma. \end{equation}
We then have the Plancherel theorem for each fixed $\tau$

\begin{equation} \label{PLANCH}  \int_{- \infty}^{\infty} | 
\gcal (\gamma_{x, \xi}^* \phi_{\lambda_j}^{\C} (s + i \tau)  |^2  ds =
 \int_{\R} e^{ - 2 |\tau | \sigma} |\nu_{\lambda_j}^{x, \xi, \gcal} (\sigma) |^2 d \sigma.
\end{equation}

As in the periodic case,  the growth rate of $|\gamma_{x, \xi}^* \phi_{\lambda_j}^{\C} (s + i \tau)  |$
as $\lambda_j \to \infty$ depends on the magnitude of the Fourier transform $\nu_{\lambda_j}^{x, \xi} (\sigma) $ for $|\sigma | \simeq \lambda_j$.

\begin{lem} \label{L2NORMB}  Let $\dim M = 2$ and suppose that $\gamma_{x, \xi}$ is a non-periodic geodesic such that QER
holds in the real domain along each finite arc (such as a uniform geodesic. )
Then for all $\epsilon > 0$, there exists $C_{\epsilon} > 0$ so that
$$\int_{ |\sigma| \geq (1 - \epsilon) \lambda_j}  |\nu_{\lambda_j}^{x, \xi, \gcal}(\sigma)|^2 d \sigma\geq  C_{\epsilon}. $$
Consequently, 
$$\int_{ |\sigma| \geq (1 - \epsilon) \lambda_j}  |\nu_{\lambda_j}^{x, \xi, \gcal}(\sigma)|^2 e^{-2\sigma \tau} d \sigma   \geq  C_{\epsilon} e^{ 2 |\tau| (1 - 
\epsilon) \lambda_j}. $$

\end{lem}

\begin{proof}

We consider  the test operator $\bar{\gcal} \chi(\lambda^{-1} D) \gcal$ in the real
domain, and its matrix elements 
$$\langle \chi(\lambda_j^{-1} D)  \gcal \gamma_{x, \xi}^* \phi_j,
\gcal \gamma_{x, \xi}^* \phi_j\rangle _{L^2(\R)}= 
\langle \overline{\gcal}\chi(\lambda_j^{-1} D)  \gcal \gamma_{x, \xi}^* \phi_j,
\gamma_{x, \xi}^* \phi_j\rangle\rangle _{L^2(\R)}$$
\begin{lem}\label{gcalWL}
There  exists a subsequence of eigenvalues $\{\lambda_{j_k}\}$ of density one so that,
\begin{equation} \label{QEG} \lim_{k \to \infty} \langle \chi(\lambda_{j_k}^{-1} D)  \gcal \gamma_{x, \xi}^* \phi_{j_k},
\gcal \gamma_{x, \xi}^* \phi_{j_k} \rangle = \int_{B^* \R} |\gcal(t)|^2 \chi(\sigma) d t d \sigma.  \end{equation}
\end{lem}

\begin{proof} 
 That is,  for any $T$, 
\begin{equation} \label{QEGa}  \lim_{k \to \infty} \langle {\bf 1}_{[-T,T]} \chi(\lambda_{j_k}^{-1} D)  \gcal \gamma_{x, \xi}^* \phi_{j_k},
\gcal \gamma_{x, \xi}^* \phi_{j_k} \rangle = \int_{B^* \R} \langle {\bf 1}_{[-T,T]}   |\gcal(t)|^2 \chi(\sigma) d t d \sigma.  \end{equation}

On the other hand, there exists a subsequence of density one so that the $\gcal$-weighted
mass on the complement is arbitrarily small, i.e. for all $\epsilon$ there exists $T(\epsilon)$
so that for $T \geq T(\epsilon),$
$$ \limsup_{k \to \infty} | \langle {\bf 1}_{|t| \geq T} \chi(\lambda_{j_k}^{-1} D)  \gcal \gamma_{x, \xi}^* \phi_{j_k},
\gcal \gamma_{x, \xi}^* \phi_{j_k} \rangle | \leq \epsilon. $$
To prove this, we consider the Weyl sums
$$N(\lambda, T, \gcal) = \sum_{j: \lambda_j \leq \lambda} 
|\langle {\bf 1}_{|t| \geq T} \chi(\lambda_{j_k}^{-1} D)  \gcal \gamma_{x, \xi}^* \phi_{j_k},
\gcal \gamma_{x, \xi}^* \phi_{j_k} \rangle|. $$
Assuming with no loss of generality that we use a positive quantization, this sum is bounded above by 

$$\begin{array}{lll} N(\lambda, T, \gcal) & \leq &  \sum_{j: \lambda_j \leq \lambda} 
\langle {\bf 1}_{|t|  \geq T} \left| \bar{\gcal } \chi(\lambda_{j_k}^{-1} D)  \gcal \right| \gamma_{x, \xi}^*  \phi_{j_k},
\gamma_{x, \xi}^* \phi_{j_k} \rangle \\ &&\\
& = &  
\int_{T^* \R}  {\bf 1}_{|t| \geq T}  \sigma \left(| \bar{\gcal }\left|\chi(\lambda_{j_k}^{-1} D)  \gcal\right|
 \right)
\sum_{j: \lambda_j \leq \lambda} d W_j^{x, \xi}  .  \end{array} $$
Then
$$\begin{array}{lll} \lim_{\lambda \to \infty} \frac{N(\lambda, T, \gcal) }{N(\lambda)}& \leq & 
\int_{B^* \R}  {\bf 1}_{|t| \geq T}  |\gcal|^2 \chi(\sigma) ds d \sigma < \epsilon,  \end{array} $$
if $T$ is chosen large enough so that $\int_{|t| \geq T} |\gcal|^2 ds < \epsilon.$ 
By definition, a  density one proportion terms of the series in the  numerator must be 
$< 2 \epsilon$.

\end{proof}


On the other hand we can use \eqref{FIFg} with $\tau = 0$ to get 
$$ \int_{\R}  \chi(\sigma/\lambda)  \left|\nu_{\lambda_j}^{x, \xi, \gcal}(\sigma) \right|^2 dt \simeq\int_{ |\sigma| \geq (1 - \epsilon) \lambda_j}  |\nu_{\lambda_j}^{x, \xi, \gcal}(\sigma)|^2 dt . $$

Combining the two limit formulae gives,  for any $\epsilon > 0$,
$$\liminf_{\lambda_j \to \infty} \int_{ |\sigma| \geq (1 - \epsilon) \lambda_j}  |\nu_{\lambda_j}^{x, \xi, \gcal}(\sigma)|^2 dt  \geq 2 \epsilon
||\gcal||^2_{L^2(\R)}. $$
In the complex domain on $\partial S_{\tau}$ we then use \eqref{PLANCH}  to get

$$\liminf_{\lambda_j \to \infty} \int_{ |\sigma| \geq (1 - \epsilon) \lambda_j}  |\nu_{\lambda_j}^{x, \xi, \gcal}(\sigma)|^2  e^{-2 \sigma \tau}  d \sigma \geq   2 \epsilon
||\gcal||^2_{L^2(\R)}. e^{2 |\tau| (1 - 
\epsilon) \lambda_j}, $$
concluding the proof of the Lemma.

\end{proof}

\subsection{Logarithmic asymptotics}

As in the periodic case, we can then compute logarithmic asymptotics of $L^2$ norms on $\partial S_{\tau}$:

\begin{lem} \label{L2NORMintrononper} Assume that $\{\phi_j\}$ satsifies QER along arcs of the
 geodesic $\gamma_{x, \xi}$. Let $||\gcal\gamma_{x, \xi}^{\tau*} \phi_j^{\C}||_{L^2(\partial S_{\tau})}$ be the $L^2$-norm
 along $\partial S_{\tau}$. Then for all $\gcal$ as above,
$$\lim_{\lambda_j \to \infty} \frac{1}{\lambda_j} \log ||\gcal \gamma_{x, \xi}^{\tau*} \phi_j^{\C}||^2_{L^2(\partial S_{\tau})}
= 2 |\tau| .$$
\end{lem}

\begin{proof}

By Lemma \ref{L2NORMB}, we have   for any $\epsilon > 0$, and any $\gcal$ as above,
$$\begin{array}{lll} \liminf_{\lambda_j \to \infty}  \frac{1}{\lambda_j}\log ||\gcal \gamma_{x, \xi}^{* \tau} \phi_j^{\C}||^2_{L^2(\partial S_{\tau})}
& = & \liminf_{\lambda_j \to \infty}  \frac{1}{\lambda_j} \log \int_{\R} e^{ - 2 \tau \sigma} |\nu_{\lambda_j}^{x, \xi, \gcal} (\sigma) |^2 d \sigma  \\ &&\\
& \geq & 2  (1 - \epsilon) |\tau| \end{array} .$$
On the other hand, the upper bound follows (as usual) by Proposition \ref{PWa}.

\end{proof}

Let $\gcal(t) = \frac{1}{t + i p}$ where $|p| $ is sufficiently large (as above).
  Lemma \ref{L2NORMB}  implies:

\begin{cor} \label{Nj}
For any $\epsilon > 0$
and any $\lambda_j$ there exists $N_j = N(\epsilon, \lambda_j) \in \R$ so that
$$ \liminf_{\lambda_j \geq \infty} \frac{1}{ \lambda_j} \log \int_{N_j}^{N_{j} + 1} |\gamma_{x, \xi}^{* \tau} \phi_j^{\C}|^2 dt > 2  |\tau| - \epsilon. $$
The choice of unit length intervals here is arbitrary. 
\end{cor}

\begin{proof}

Using obvious upper bounds on $\gcal$ on the intervals $[n, n + 1]$ we have,
$$  \int_{- \infty}^{\infty} | 
\gcal (\gamma_{x, \xi}^* \phi_{\lambda_j}^{\C} (s + i \tau)  |^2  ds 
\leq \sum_{n \in \Z} \frac{1}{1 + n^2}  \int_{n + 1}^{n} |\gamma_{x, \xi}^{* \tau} \phi_j^{\C}|^2 dt. $$
Hence by Lemma \ref{L2NORMB} , 
$$ 2 |\tau| \leq  \liminf_{\lambda_j \geq \infty}  \frac{1}{ \lambda_j} \log  \sum_{n \in \Z} \frac{1}{1 + n^2}  \int_{n + 1}^{n} |\gamma_{x, \xi}^{* \tau} \phi_j^{\C}|^2 dt. $$ If the Lemma were false, we would have for all $n$, and sufficienty
large $\lambda_j$,
$$  \int_{n}^{n} |\gamma_{x, \xi}^{* \tau} \phi_j^{\C}|^2 dt <  e^{2 |\tau| - \epsilon} ,$$
and then $  \frac{1}{ \lambda_j} \log  \sum_{n \in \Z} \frac{1}{1 + n^2}$ of these integrals would be
$\leq 2 |\tau| - \epsilon,$ a contradiction.

\end{proof}


For any set $\{a_{I, j}, j \in I\}$ of real numbers, 
$$\max_{j \in I} a_{I, j} \leq \frac{1}{N} \log \sum_{j \in I} e^{N a_{I, j} } \leq \max_{j \in K} a_{I,j} + \frac{1}{N} \log \# I. $$
The analogous bounds for integrals when $|\gcal|^2 dt$ is  a probability measure is,
$$\begin{array}{lll} \sup_{t \in \R}   \frac{1}{ \lambda_{j_k}}
 \log \left| \gamma_{x, \xi}^* \phi_{\lambda_{j_k}}^{\C} (t + i \tau)
\right|^2  & \leq  & \frac{1}{ \lambda_{j_k}}
 \log \int_{\R} \left| \gamma_{x, \xi}^* \phi_{\lambda_{j_k}}^{\C} (t + i \tau)
\right|^2  |\gcal|^2 dt \\ \\ & &  \leq 
\sup_{t \in \R}   \frac{1}{ \lambda_{j_k}}
 \log \left| \gamma_{x, \xi}^* \phi_{\lambda_{j_k}}^{\C} (t + i \tau)
\right|^2 +  \frac{1}{ \lambda_{j_k}}
 \log \int_{\R} |\gcal|^2 dt.  \end{array}$$

\subsection{Completion of proof of Theorem \ref{MAINCORnonper}}

As in the previous cases, we need to rule out 
\eqref{SMALL}.  

By Lemma \eqref{L2NORMintrononper} and Corollary \ref{Nj} there exists  a sequence $[N_j, N_{j + 1}]$
for which the lower bound of Corollary \ref{Nj} holds.  For this sequence, \eqref{YBAD}  is false. It follows
by Lemma \ref{HARTOGSM} that $v_j \to |\tau|$. This completes the proof of  Theorem  \ref{MAINCORnonper}.

\section{\label{APPENDIX} Appendix}

In this appendix, we review the QER result of \cite{TZ} (see also \cite{DZ}). We also review the theory of Fourier integral operators with complex phase that
we use in this article.   We refer to \cite{MSj} and volume IV of \cite{Ho} for background. Since the
manifolds and metrics in this article are real analytic, the theory of almost analytic extensions is not
needed.

\subsection{\label{QER} Quantum ergodic restriction in the real domain}

In this section we review the QER theorem for hypersurfaces of
\cite{TZ}. 
There is no advantage to specializing to curves in surfaces, so we
review the result for hypersurfaces $H \subset M$.

Let $H \subset M$ be an embedded submanifold, and denote by
\begin{equation} T^*_{H} M = \{(q, \xi) \in T_q^* M, \;\; q\in H\} \end{equation}
the cotangent bundle of the ambient space along $H$.  We also
denote by  $T^* H=  \{(q, \eta) \in T_q^* H, \;\; q\in H\}$ the
cotangent bundle of $H$.  We further denote by $r_{H} : T^*_H M
\to T^* _H M$ the reflection map through $T^* H$,  i.e.  $r_H(\xi) = \xi'$ with $\xi|_{TH} = \xi' |_{TH}$
but with opposite normal components.

 We define the {\it first return time}
$T(s, \xi)$ on $S^*_H M$ by,
\begin{equation} \label{FRTIME} T(s, \xi) = \inf\{t > 0: G^t (s,
\xi) \in S^*_H M, \;\;\ (s, \xi) \in S^*_H M)\}. \end{equation} By
definition $T(s, \xi) = + \infty$ if the trajectory through $(s,
\xi)$ fails to return to $H$.  We  define the first return map on
the same domain by
\begin{equation} \label{FIRSTRETURN} \Phi: S^*_H M \to S^*_H M, \;\;\;\; \Phi(s, \xi) = G^{T(s, \xi)} (s, \xi) \end{equation}
When $G^t$ is ergodic, $\Phi$ is defined almost everywhere and is
also ergodic.


\begin{maindefin} \label{ANC}  We say that  $H$ has a positive
measure of microlocal reflection symmetry if 
$$  \mu_{L, H} \left( \bigcup_{j \not= 0}^{\infty}  \{(s, \xi) \in S^*_H M : r_H G^{T^{(j)}(s, \xi)} (s, \xi)  =
 G^{T^{(j)}(s, \xi)} r_H (s, \xi)  \}\right) > 0.  $$ 
Otherwise we say that $H$ is asymmetric with respect to the geodesic flow. 

\end{maindefin}

To state the QER result for Dirichlet data, we need some further notation. The result holds for several 
classes of pseudo-differential operators on $H$ with essentially the same proof. We state the result
first for pseudo-differential operators with classical poly-homogeneous symbols 
 $$ a(s,\sigma) \sim \sum_{k=0}^{\infty} a_{-k}(s,\sigma), \,\,(a_{-k} \; \mbox{ positive homogeneous of order} 
\; -k) $$
on $T^* H$ and then
for semi-classical pseudo-differential operators with semi-classical  symbols  $a \in S^{0,0}(T^*H \times (0,h_0]$
of the form
 $$ a_{\hbar} (s,\sigma) \sim \sum_{k=0}^{\infty} \hbar^k \;a_{-k}(s,\sigma), \,\,(a_{-k} \; \in  S_{1,0}^{-k}(T^* H)) $$ as in
\cite{HZ,TZ}.  

The restriction map $S^*_H M \to B^* H$ is singular along $S^* H$ and pushes forward Liouville measure
to a  multiple  $\gamma_{B^*H}^{-1} ds d \sigma$ of the symplectic volume density on $B^* H$. Here,
$\gamma_{B^* H} : = (1 - |\sigma|^2)^{\half}$ We note that $\gamma$ is a zeroth-order homogeneous function
on $T^*_H M$ which equals the non-homogeneous $\gamma_{B^*H}$ 
 of \cite{HZ}  on $S^*_H M$ (i.e.  $|\eta_n|^2 + |\sigma|^2 = 1$).

For homogeneous pseudo-differential operators, the QER theorem is as follows:

   \begin{theo} \label{maintheorem} Let $(M, g)$ be a compact manifold with ergodic geodesic flow, and let  $H \subset
      M$ be a hypersurface.  Let $\phi_{\lambda_j}; j=1,2,...$ denote the
       $L^{2}$-normalized eigenfunctions of $\Delta_g$.
       If $H$ has a zero measure of microlocal symmetry, then
 there exists a  density-one subset $S$ of ${\mathbb N}$ such that
  for $\lambda_0 >0$ and  $a(s,\sigma) \in S^{0}_{cl}(T^*H)$
$$ \lim_{\lambda_j \rightarrow \infty; j \in S} \langle  Op(a) 
 \phi_{\lambda_j}|_{H},\phi_{\lambda_j}|_{H} \rangle_{L^{2}(H)} = \omega(a), $$
 where 
 $$  \omega(a) = \frac{4}{ vol(S^*M) } \int_{B^{*}H}  a_0( s, \sigma )  \,  \gamma^{-1}_{B^*H}(s,\sigma)  \, ds d\sigma. $$

\end{theo}

The analogous result for semi-classical pseudo-differential operators is:

  \begin{theo} \label{sctheorem} Let $(M, g)$ be a compact manifold with ergodic geodesic flow, and let  $H \subset
      M$ be a hypersurface.  
       If $H$ has a zero measure of microlocal symmetry, then
 there exists a  density-one subset $S$ of ${\mathbb N}$ such that
  for $a \in S^{0,0}(T^*H \times [0,h_0)),$
$$ \lim_{h_j \rightarrow 0^+; j \in S} \langle Op_{h_j}(a)
 \phi_{h_j}|_{H},\phi_{h_j}|_{H} \rangle_{L^{2}(H)} = \omega(a), $$
 where 
 $$  \omega(a) = \frac{4}{ vol(S^*M) } \int_{B^{*}H}  a_0( s, \sigma )  \,  \gamma^{-1}_{B^*H}(s,\sigma)  \, ds d\sigma.$$
\end{theo}

\subsection{\label{APPENDIXFIOCX} Fourier integral distributions  with complex phase}

First,  we  review  the relevant definitions (see \cite{Ho} IV, \S 25.5 or \cite{MSj}).  A Fourier integral distribution
with complex phase on a manifold $X$ is a distribution that can locally be represented by an oscillatory integral
$$A(x) = \int_{\R^N}  e^{i \phi(x, \theta)} a(x, \theta) d \theta$$
where $a(x, \theta) \in S^m( X \times V)$ is a symbol of order $m$ in a cone $V \subset \R^N$ and where the phase $\phi$ is a positive
regular phase function, i.e. it satisfies
\begin{itemize}

\item $\Im \phi \geq 0$;

\item $d\frac{\partial \phi}{\partial \theta_1}, \dots, d \frac{\partial \phi}{\partial \theta_N}$ are
linearly independent complex vectors on $$C_{\phi \R} = \{(x, \theta) : d_{\theta} (x, \theta) = 0\}. $$

\item In the analytic setting (which is assumed in this article), $\phi$ admits an analytic continuation 
$\phi_{\C}$ to an open
cone in 
$X_{\C} \times V_{\C}$.

\end{itemize}

Define  
$$C_{\phi_{\C}} = \{(x, \theta) \in X_{\C} \times V_{\C} : \nabla_{\theta} \phi_{\C}(x, \theta) = 0\}. $$
Then $C_{\phi_{\C}}$ is a manifold near the real domain. One defines the Lagrangian submanifold
$\Lambda_{\phi_{\C}} \subset T^* X_{\C}$ as the image 
$$(x, \theta) \in C_{\phi_{\C}} \to (x, \nabla_x \phi_{\C}(x, \theta)). $$

According to Definition 4.4 of \cite{MSj}, the space $I^m(X, \Lambda)$ of Fourier integral operators of order $m$
with complex phase is the class of operators satisfying
\begin{itemize}

\item $WF(A) \subset \Lambda_{\R}$;

\item For every $\lambda_0 \in \Lambda_{\R}$ and every choise of local coordinates $x_1, \dots, x_n$
near $\pi(\lambda_0)$, $A$ is microlocally of the form $I(a, \phi)$ near $\lambda_0$ where $\phi$ is a positive
phase function generating $\Lambda$ near $\lambda_0$ and where $a \in S^{m + (n + 2 N)/4}(\R^n \times \R^N)$
has its support in a small conic neighborhood of $(x_0, \theta_0) \in C_{\phi \R}$, i.e. the point corresponding
to $\lambda_0$. 

\end{itemize}

Given a closed conic positive Lagrangian manifold $\Lambda \subset T^* \tilde{ X} - 0$, there
exists a  principal symbol map
$$ A \in I_c^n(X, \Lambda)/ I_c^{m-1}(X, \Lambda) \to \Gamma^{m + n/4}(\Lambda; \lcal), $$
and also a quantization (denoted $\pcal$ in \cite{MSj}) which inverts it. As in the real domain, given a
real analytic phase $\phi$ and its holomorphic extension $\tilde{\phi}$ parametrizing $\Lambda$, one defines
the Leray residue form $d_{\tilde{\phi}}$ on $C_{\tilde{\phi}} $  by
$$d_{\tilde{\phi}} \wedge d \frac{\partial \tilde{\phi}}{\partial \tilde{\theta_1}} \wedge \cdots \wedge d \frac{\partial \tilde{\phi}}{\partial \tilde{\theta_n}} = i^{n + N} d z_1, \dots \wedge d z_n \wedge d \tilde{\theta}_1 \wedge \cdots \wedge
d \tilde{\theta}_N. $$
If $I(\phi, A)$ is a complex oscillatory integral with positive phase, and $a_0$ is the principal term of the
amplitude $A$, then the symbol of $I(\phi, A)$ is
$$a_0 \sqrt{d_{\tilde{\phi}} }. $$

\subsection{\label{CCRPITAU} Complex canonical relation of $\Pi_{\tau}$}

The complex canonical relation
of $\Pi_{\tau}$, which lies in the complex co-tangent bundle of the (Cartesian square of the) complexification $\tMpt$ of $\partial M_{\tau}$.
The positive complex canonical relation of $\Pi_{\tau}$ is the idempotent canonical relation 
$$ C_{\tau} \subset T^*(\tMpt\times \tMpt) $$ satisfying
$C^2 = C = C^*$ given by
\begin{equation} \label{Ctau}  C_{\tau} =\{ (z, \theta  \partial_z \rho, w, \theta \dbar_w \rho): z, w \in \tMpt , 
\rho(z, w) = \tau. \}. \end{equation}
Thus, if we put
$$S_{\tau} = \{(z,w) \in \tMpt \times \tMpt: \tilde{\sqrt{\rho}} (z, w) = \tau\} $$
then  $$C_{\tau} = N^* S_{\tau}. $$

The canonical relatin $C_{\tau}$ can also be described as a flowout relation  in terms of complex characteristics of the
tangential Cauchy-Riemann operator $\dbar_b$. 
As a strongly pseudo-convex hypersurface in the complex manifold $M_{\C}$,
$\partial M_{\tau}$ is a CR manifold whose complexified tangent bundle has a complex codimension one
subspace  invariant under the complex structure $J$. 
We denote by $Z_1, \dots, Z_n$, resp.
$\bar{Z}_1,
\dots, \bar{Z}_m$ an orthonormal basis with respect to the \kahler form $\omega_{\rho}$ on $M_{\C}$
of the holomorphic tangent space $T^{1,0} \partial M_{\tau}$,  resp. the anti-holomorphic tangent
space $T^{0,1}(\partial M_{\tau})$. Then $\Box_b = \sum_j \bar{Z}_j^* \bar{Z}_j$.

We denote the symbol of $\Box_b$ by  $q$.
Its zero set is the characteristic variety $\Sigma_{\tau}$ of $\Box_b$ in the real cotangent space $T^* \partial M_{\tau}$, i.e. simultaneous
kernel of the functions
\begin{equation} \zeta_j(x, \xi) =
\langle \xi, \bar{Z}_j \rangle, \end{equation} 
which are the symbols of the associated derivative along $\bar{Z}_j$. Thus,
$$q = \sum_{j = 1}^d |\zeta_j|^2 : T^* \partial M_{\tau} \to \R. $$  When we holomorphically extend to
$\tMpt$, we get the complex characteristic variety $\jcal_+ \subset T^* (\tMpt)$, the zero set
of $\tilde{q}$, the holomorphic extension of $q$.
We let $\tilde{\zeta}_j$ be the analytic extensions to $\tMpt$  of the functions
$\zeta_j$ and $\tilde{\sigma}$ be the standard  holomorphic symplectic form of
$T^* \tMpt$.  Thus,
\begin{equation} \mathcal{ J}_+ = \{(\tilde{x}, \tilde{\xi}) \in T^*
\tMpt: \tilde{\zeta}_j = 0\  \forall j\}\; = \{\tilde{q} = 0\}.  \end{equation}
It is an  involutive sub-manifold of $T^*\tMpt$ with the properties:
\begin{equation}\label{JCAL}  \begin{array}{rl} \mbox{(i)} & (\jcal _+)_{\R} = \Sigma \\ &  \\\mbox{(ii)} &
\frac{1}{i} \tilde{\sigma} (u, \bar{u}) > 0, \forall u \in T(\jcal _+)^{\bot}\\ &
\\\mbox{(iii)} &
T_{\rho}(\jcal _+) = T_{\rho} \tilde{\Sigma} \oplus W_{\rho}^+.
\end{array} \end{equation}
Here, $W_{\rho}^+$ is the sum of the eigenspaces of $F_{\rho}$,
the normal Hessian
of $q$,  corresponding to the eigenvalues
$\{ i \lambda_j\}$ with $\lambda_j \geq 0$.  Thus, $\jcal_+$ is the stable manifold for the Hamiltonian flow of $\tilde{q}$ on 
$T^* \tMpt$. 

Since $\jcal_+$ is a co-isotropic (i.e. involutive)  submanifold of $T^* \tMpt$, it has
a null folation, which   is given by
the joint Hamilton flow of the defining functions $\tilde{\zeta}_j$. 
We then define 
\begin{equation} \label{CTAUDEF} C_{\tau}: = \{(\tilde{x}, \tilde{\xi}, \tilde{y}, \tilde{\eta}) \in
\jcal _+ \times \overline{\jcal _+}: (\tilde{x}, \tilde{\xi}) \sim
(\tilde{y}, \tilde{\eta})\}, \end{equation}
where $\sim$ is the equivalence relation of `belonging to the same leaf of
the null foliation of $\jcal _+.$

This equivalence relation may be described in terms of Hamilton flows.   One has a fibration
$$\pi_+: \jcal_+ \to \Sigma$$
whose fiber at $\sigma$ is
 the orbit  of $\sigma$ under   the joint Hamilton
flow of the $\tilde{\zeta}_j$'s.  Then
$$C = \jcal_+ \times_{\pi_+}  \overline{\jcal_+}. $$
Equivalently, $C$ is the flow-out of $\jcal_+ \oplus \overline{\jcal_+}$ from $\Delta_{\Sigma \times \Sigma}$.

It is clear from the description that $C
\circ C = C^* = C,$ i.e. that $C$
is an idempotent canonical relation.
The following proposition,  proved in \cite{MSj} and in
(\cite{BoGu}), Appendix, Lemma 4.5).

\begin{prop}  \label{UNIQUE} $C_{\tau}$ is the unique strictly positive almost analytic
canonical relation
$C$
satisfying
$$ \diag(\Sigma) \subset C \subset \jcal _+ \times
\overline{\jcal _+}.$$  \label{C}\end{prop}


\end{document}